\documentclass[11pt]{amsart}
\usepackage{mathrsfs}
\usepackage{graphicx}
\usepackage{amscd}
\usepackage{amsmath}
\usepackage{latexsym}
\usepackage{amsfonts}
\usepackage{amssymb}

\newtheorem{theorem}{Theorem}[section]
\theoremstyle{plain}

\newtheorem{corollary}{Corollary}[section]

\newtheorem{definition}{Definition}[section]

\newtheorem{lemma}{Lemma}[section]

\newtheorem{proposition}{Proposition}[section]
\newtheorem{remark}{Remark}[section]

\numberwithin{equation}{section}

\newcommand\BR{{\mathbb {R}}}
\newcommand\vep{{\varepsilon}}

\newcommand\ep{{\epsilon}}

\newcommand\pd{{\partial}}

\newcommand\bes{\begin{eqnarray}}
\newcommand\ees{\end{eqnarray}}
\newcommand\bess{\begin{eqnarray*}}
\newcommand\eess{\end{eqnarray*}}

\newcommand\Dl{{\Delta}}

\title[]{}

\author{Guangyu Zhao and Shigui Ruan}
\address{Department of Mathematics \\
University of Miami, Coral Gables, FL 33146}
\email[]{gzhao@math.miami.edu, ruan@math.miami.edu}

\title[]
{The decay rates of Traveling Waves for a class of Nonlocal
Evolution Equations}

\begin{document}

\begin{abstract}  We obtain the precise decay rates of traveling wave for a class of
nonlocal evolution equations arising in the theory of phase
transitions. We also investigate the spectrum of the operator
obtained by linearizing at such a traveling wave. The detailed
description of the spectrum is established.
\end{abstract}
\maketitle

\setlength{\baselineskip}{17.5pt}

\newcommand\Xtheta{{\mathcal X}}
\newcommand\bx{{\mathbf x}}
\newcommand\bi{{\mathbf i}}
\newcommand\bt{{\small \mathbf T}}
\newcommand\bn{{\small \mathbf N}}
\newcommand\bI{{\mathbf I}}
\newcommand\curv{{K}}
\newcommand\La{{\Lambda}}
\newcommand\clo{{{\mathcal L}^\pm_0}}
\newcommand\clomp{{{\mathcal L}^\mp_0}}
\newcommand\cls{{{\mathcal L}^s}}
\newcommand\cle{{{\mathcal L}_\vep}}
\newcommand\ltwo{{{\bf X}_0}}
\newcommand\tQ{{Q}}
\newcommand\tu{{\tilde u}}
\newcommand\al{{\alpha}}
\newcommand\cl{{{\mathcal L}_{\vep,\delta}^+}}
\newcommand\clpm{{{\mathcal L}_{\vep,\delta}^\pm}}
\newcommand\clmp{{{\mathcal L}_{\vep,\delta}^\mp}}
\newcommand\ak{{\alpha_k}}
\newcommand\ank{{\alpha_{-k}}}
\newcommand\ue{{u_\vep}}
\newcommand\un{{u_n}}
\newcommand\phie{{\phi_\vep}}
\newcommand\ce{{c_\vep}}
\newcommand{\rar}{\rightarrow}

\newcommand\re{{\mathcal R}}
\newcommand\nonl{{\mathbf N}}
\newcommand\mapT{{\mathbf T}}
\newcommand\real{{\hbox{\rm Re}}}

\section{Introduction}

In this paper, we are concerned with a class of nonlocal evolution
equations of the form
\begin{equation}\label{te1}
\frac{\pd u(x,t)}{\pd t}=d\frac{\pd^2 u(x,t)}{\pd
x^2}+f(u(x,t),(J*u)(x,t))
\end{equation}
for $x\in\BR$ and $t\in\mathbb{R}^+$. Here $d\geq 0$ is a constant,
$(J*u)(x,t):=\int_{\mathbb{R}}J(x-y)u(y,t)dy$, $f$ and $J$ are
sufficiently smooth functions. Depending upon the constant $d$ and
the nonlinearity $f$ involved, equation (\ref{te1}) may model the
spatio-temporal development of various populations or epidemics (
see the surveys and references cited therein). Similar equations
have been also derived and studied from the point of view of certain
continuum limits in the dynamic Ising models (see \cite{bates1},
\cite{bates2}, \cite{bates3}, \cite{bates4}, \cite{bates5} and
references therein). Equation (\ref{te1}) has received much
attention recently, the possible interest of such an equation lies
in the fact that much more general types of interactions in the
medium can be account for. The existence as well as the uniqueness
of a traveling wave solution for integro-differential equations
(\ref{te1}) have been of great interest, both from a mathematical
standpoint and for their applications. Indeed, our study of
(\ref{te1}) is motivated by the following traveling wave problems.

A. Family of neurons
\[
\frac{\pd u}{\pd t}=-u+\int_{\mathbb R}J(x-y)S(u(y,t))dy,
\]
where $m(u):=S(u)-u$ satisfies $m'(0)<0,m'(1)<0$. $J$ is a smooth
kernel satisfying
\begin{equation}\label{tezh1}
J\geq0\,\ \text{on}\,\ \mathbb{R},\,\ \int_{\mathbb{R}}J=1.
\end{equation}

B. Ising model
\[
\frac{\pd u}{\pd t}=\text{tanh}{\beta(J*u+h)}-u,
\]
where $\beta>1$ is inverse temperature and $h$ is a constant. $J$ is
a smooth kernel supported in $[-1,1]$ satisfying (\ref{tezh1}).

C. Phase transition
\[
\frac{\pd u}{\pd t}=\varepsilon[J*u-u]+g(u),
\]
where $g(u)$ is a bistable function, $\lambda>0.$

Throughout this paper, we make the following hypotheses.

(H1) $d+|c|\neq 0.$

(H1) $J\in C(\BR)$ is even, nonnegative such that
\[\int_{\BR}J(s)ds=1 \,\,\ \text{and}\,\,\,\ \int_{\BR}J(s)e^{\rho
s}ds<+\infty\,\,\,\,\   \text{for any}\,\ \rho\in\BR.\]

(H2) $f\in C^{2,\alpha}(\BR\times\BR)$ and
$f(-1,-1)=f(1,1)=f(q,q)=0$, where $-1<q<1$.

(H3) $\pd_{s}f(r,s)>0$ for any fixed $(r,s)\in [-1,1]\times[-1,1]$.

(H4) $\pd_{r}f(\pm1,\pm1)<0$ and
$\pd_{r}f(\pm1,\pm1)<-\pd_{s}f(\pm1,\pm1)$.

(H5) $\overline{f}(\cdot)=f(\cdot,\cdot)$ is bistable, i.e.
$\overline{f}$ has exactly three zeros $\pm 1$ and $q$. There exists
an interval $[l,l']\subset(-1,1)$ such that $q\in[l,l']$,
$\overline{f}'(s)\geq 0$ for any $s\in[l,l']$ and
$\overline{f}'(s)\leq 0$ for any $s\in[-1,1]\setminus[l,l']$.

Under conditions (H1)-(H5), it is well known that equation
(\ref{te1}) possesses a unique monotone traveling wave solution
connecting the equilibria $\pm 1$ (i.e solutions of the form
$u(x,t)=U(x+ct)$ for some velocity
$c,\,\lim_{\xi\rightarrow\pm\infty}U(\xi)=\pm 1$ with $\xi=x+ct$.)
However, the precise rates at which $U$ approaches the two
homogeneous equilibria $\pm 1$ are still lacking. In this paper, we
address this issue. Our main goal is to obtain the exact decay rates
of traveling wave of (\ref{te1}) as $\xi\rightarrow\pm\infty$. With
the right rates of convergence, we can easily establish the
uniqueness of the traveling wave. Recently, a spectral analysis of
traveling waves of (\ref{te1}) was made in \cite{bates2}. The
authors considered the operator obtained by linearizing (\ref{te1})
at $U$ in $C_0(\BR)$, the space of continuous functions which vanish
at infinity. They show that the operator has spectrum in the left
half plane, bounded away from the imaginary axis except for an
algebraically simple eigenvalue at zero. This fact is of crucial
importance, which not only implies the exponential asymptotic
stability of traveling waves but also leads to the description of
dynamics of the codimension-one invariant stable manifolds. Here the
codimension-one invariant stable manifolds are transverse to the
one-dimensional manifold formed by the translates of the traveling
wave. Based on our study of asymptotical behavior of traveling
waves, we are able to obtain a detailed description of the spectrum
of the operator in the underlying $L^{p}$ space $(1\leq
p\leq\infty)$.

The paper is organized as follows: In section 2, we investigate the
exponential decay rates of the traveling wave and prove its
uniqueness. In section 3, we study the spectrum of the operator
obtained by linearizing (\ref{te1}) at the traveling wave. Some of
the results needed in the spectral analysis can be obtained by using
arguments similar to those in \cite{Mall}, and therefore we
summarize these results in the Appendix with sketched proofs that
are necessary for our purposes.

\section{decay rates of traveling waves}
In this section, we study the asymptotical behavior of traveling
wave $(c,U)\in\BR\times C^{2}(\BR)$ which satisfies
\begin{equation}\label{te3}
\left\{\begin{array}{ll}
cU'=dU''+f(U,J\ast U)\,\,\,\, \text{on}\,\,\ \BR,\\
\lim_{\xi\rightarrow\pm\infty}U(\xi)=\pm 1,\,\,\ U'>0 \,\
\text{on}\,\ \mathbb{R}.
\end{array}\right.
\end{equation}
We show that the behavior of the traveling wave near $\pm\infty$ is
governed by exponentials. Moreover, we determine the exact
exponential decay rates of $U$ as $\xi\rightarrow\pm\infty$. For our
purpose, we shall adapt the Fourier transform techniques presented
in \cite{Mall} and \cite{Pazy} (see also \cite{zhao1} and
\cite{zhao2}). By differentiating equation (\ref{te3}) with respect
to $\xi$, we obtain
\begin{equation}\label{teb1}
c(U')'=d(U')''+f_r(U,J\ast U)U'+f_s(U,J\ast U)J\ast U'.
\end{equation}
From (\ref{te3})
\begin{equation}\label{teb2}
\lim_{\xi\rightarrow\pm\infty}f_r(U,J\ast U)=f_r(\pm1,\pm1),\,\,\
\lim_{\xi\rightarrow\pm\infty}f_s(U,J\ast U)=f_s(\pm1,\pm1).
\end{equation}
Motivated by (\ref{teb1}) and (\ref{teb2}), we consider the linear
operator $L:D(L)\subset L^{p}(\BR, \mathbb{C})\rightarrow L^{p}(\BR,
\mathbb{C})$ defined by
\begin{equation}\label{teb3}
Lv:=dv''-cv'+a(\xi)v+b(\xi)J\ast v,\,\,\ \xi\in\mathbb{R},
\end{equation}
where $D(L):=\{v\in L^{p}|v',dv''\in L^p\}$, $a, b\in
C_{b}(\mathbb{R},\mathbb{R}),$ and $1\leq p\leq \infty.$

A special case occurs if both $a$ and $b$ are constants, we define
$L_0:D(L_0)\subset L^{p}(\BR, \mathbb{C})\rightarrow L^{p}(\BR,
\mathbb{C})$ by
\begin{equation}\label{te8}
L_0v:=dv''-cv'+av+bJ\ast v,\,\,\ \xi\in\mathbb{R},
\end{equation}
where $D(L_0):=D(L)$

In what follows, when convenient, $f_r(\pm1,\pm1),f_s(\pm1,\pm1)$
are denoted by $a^\pm$ and $b^\pm$, respectively.

Let $\Delta_{0}(z):{\mathbb {C}}\rightarrow{\mathbb{C}}$ be the
characteristic function associated with $L_0$, defined by
\[
\Delta_{0}(z)=dz^2-cz+a+b\int_{\BR}J(s)e^{-zs}ds.
\]

In an attempt to solve the inhomogeneous equations
\begin{equation}\label{dez5}
L_{0}v=h,\,\,\,\ h\in L^p,
\end{equation}
we may formally take the Fourier transform to obtain
\[
\Dl_{0}(i\eta)\widehat{v}(\eta)=\widehat{h}(\eta),\,\,\
\eta\in\mathbb R,
\]
where $\widehat{g}(z)=(2\pi)^{-1}\int_{\mathbb R}g(s)e^{-izs}ds,\,\
i=\sqrt{-1}$ and $z\in\mathbb C$. Note that
$\Delta_{0}^{-1}(i\eta)=O(|\eta|^{-1})$. Therefore, we can take the
inverse transform of $\Delta_{0}^{-1}(i\eta)$ to obtain solution $v$
provided $\Delta_{0}(i\eta)\neq 0$ for any $\eta\in\BR$.

\begin{definition}
The operator $L_0$ is called hyperbolic if $\Delta_0(i\eta)\neq0$
for any $\eta\in\BR.$
\end{definition}

In what follows, for a given complex number $z\in\mathbb{C}$, we
shall always denote its real part and imaginary part by $\Re z$ and
$\Im z$, respectively. The following Lemma ensures the existence of
$\Delta_{0}^{-1}(i\eta)$ under suitable conditions.
\begin{lemma}\label{tep1}
Suppose that $a<0,b>0$ and $b<-a$. Then

(a) The equation $\Dl_{0}(z)=0$ has precisely two real solutions
$\lambda^{s}<0<\lambda^{u}.$

(b) The zeros of $\Dl_{0}(z)$ in the vertical strip
$\{z\in\mathbb{C}|\lambda^{s}\leq\Re z \leq\lambda^{u}\}$ are
$\lambda^{s}$ and $\lambda^{u}$. In addition, in each vertical strip
$|\Re z|\leq K$, there lie only finite number of zeros of
$\Dl_{0}(z)$.
\end{lemma}
\begin{proof}  Set $N(z)=b\int_{\BR}J(s)e^{-zs}ds$
and $D(z)=cz-dz^2-a$. Then $z$ is a zero of $\Dl_{0}(z)$ if and only
if $z$ is a solution to the equation $N(z)=D(z)$. Due to the
assumption, $D(0)-N(0)>0$. We start with the case that $z$ takes on
real values, note that $N(z)$ is a positive convex even function of
$z$ with $\frac{\pd^{2} N}{\pd z^2}>0$ for any $z\neq 0$.
Consequently, there exists two real roots of $N(z)=D(z),$ denoted by
$\lambda^{s}$ and $\lambda^{u}$ with $\lambda^{s}<0<\lambda^{u}.$ In
addition, it is easy to see that $N(z)<D(z)$ if
$z\in(\lambda^{+},\lambda^{u}),$ whereas, $N(z)>D(z)$ if
$z\in\mathbb{R}\setminus(\lambda^{+},\lambda^{u}).$ This confirms
the part $(a)$. Next let $z\in\mathbb{C}.$ Note that
\begin{equation}\label{te15}
|N(z)|\leq\int_{\BR}J(s)e^{-\Re zs}ds<\Re D(z),\,\,\ \forall
z\in\{z\in\mathbb{C}|\lambda^{s}<\Re z<\lambda^{u}\}.
\end{equation}
Moreover, we observe that
\[
\Re D(\lambda^{s}+i\mu)\geq\int_{\BR}J(t)e^{-\lambda^{s}t}dt>
\int_{\BR}J(t)e^{-\lambda^{s}t}\cos\mu tdt= \Re N(\lambda^{s}+i\mu)
\]
and
\[
\Re D(\lambda^{u}+i\mu)\geq\int_{\BR}J(t)e^{-\lambda^{u}t}dt>
\int_{\BR}J(t)e^{-\lambda^{u}t}\cos\mu tdt=\Re N(\lambda^{u}+i\mu)
\]
for any $\mu\neq 0$. This yields the first part of (b). Due to the
first inequality in (\ref{te15}), $|N(z)|$ is bounded in the
vertical strip $|\Re z|\leq K,\,\,\ K>0$. Clearly, when restricted
to such a strip, the solution set of $D(z)=N(z)$ is bounded. Since
$\Dl_{0}(z)$ is a entire function over $\mathbb{C}$, there are only
finitely many roots of $\Dl_{0}(z)$ in such a strip.
\end{proof}

\begin{lemma}\label{tte1}
Assume that the operators $L_0$ defined by (\ref{te8}) is
hyperbolic. Then for each $1\leq p\leq\infty$,
$L_0:D(L_0)\rightarrow L^{p}$ is an isomorphisms. The inverses is
given by convolution
\[
(L_{0}^{-1}h)(\xi)=(G_{0}\ast
h)(\xi)=\int_{\BR}G_0(\xi-\eta)h(\eta)d\eta
\]
with the functions $G_0$, which enjoy the estimate
\begin{equation}\label{te9}
|G_{0}(\xi)|\leq C e^{-\alpha|\xi|},\,\,\ \xi\in\mathbb{R}
\end{equation}
for some positive constants $C$ and $\alpha$. In particular, for
each $h\in L^{p}$, $u=L_{0}^{-1}h$ is the unique solution to the
inhomogeneous equation (\ref{dez5}).
\end{lemma}
\begin{proof}
Set
\begin{equation}\label{teza1}
G_{0}(\xi)=\int_{\mathbb
R}e^{i\xi\eta}\Delta_{0}^{-1}(i\eta)d\eta,\,\,\,\ \xi\in\mathbb R.
\end{equation}
By a similar argument used in \cite{Mall}, we may interpret $G_{0}$
as a tempered distribution and show that
\begin{equation}\label{tezj2}
dG''_{0}(\xi)-cG'_{0}(\xi)+aG_{0}(\xi)+bJ\ast
G_{0}(\xi)=\delta(\xi),
\end{equation}
where $\delta$ denotes the delta function distribution. Therefore,
when $d=0$, as a function, $G_{0}$ is absolutely continuous for all
$\xi\neq 0$ and satisfies
\[
cG'_{0}(\xi)=aG_{0}(\xi)+bJ\ast G_{0}(\xi) \,\,\ \text{almost
everywhere}.
\]
Furthermore, the function $G_{0}$ possess left- and right-hand
limits $G_{0}(0-)$ and $G_{0}(0+)$ at $\xi=0$, and there is a jump
discontinuity
\[
G_{0}(0+)-G_{0}(0-)=1.
\]
If $d>0$, then $G_{0}$ is absolutely continuous for all $\xi$ and
$G'_{0}$ is discontinuous at $\xi=0$.

We now show that the function $G_{0}$ decays exponentially at
$\pm\infty$. First observe that
\[
|\int_{\BR}J(s)e^{-zs}ds|\rightarrow 0,\,\,\ |\Im
z|\rightarrow\infty
\]
in each vertical strip $|\Re z|\leq K$.  We then have
$$\Delta_{0}(z)=dz^2-cz+O(1),\,\,\,\,\ |\Im z|\rightarrow\infty$$
uniformly in such a strip. Thanks to the assumption that
$\Delta_{0}(i\eta)\neq0$ for any $\eta\in\mathbb{R},$ Lemma
\ref{tep1} implies that there exists $\alpha>0$ such that
$\Dl^{-1}_{0}(z)$ is analytic in the strip $|\Re z|<\alpha.$  In
order to obtain (\ref{te9}), we distinguish between two cases.

The case that $d=0$.

We write
\[
\Dl^{-1}_{0}(z)=[-c(z+k)]^{-1}+R(z),
\]
where $k>2\alpha$. Clearly, in the strip $|\Re z|<\alpha$, $R(z)$ is
analytic. Moreover, $R(z)$ satisfies $R(z)=O(|\Im z|)^{-2}$
uniformly as $|\Im z|\rightarrow\infty$. Consequently, if $\xi\geq
0$, then we can calculate the function $G_{0}$ by shifting the path
of integration of integral in (\ref{teza1}) as follows:
\begin{eqnarray*}
G(\xi)&=&\frac{1}{2\pi}\int_{\BR}e^{i\xi s}\Delta_{0}^{-1}(is)ds
=\frac{1}{2\pi}\int_{\BR}e^{i\xi s}[-c(is+k)^{-1}+R(is)]ds\\
&=&-\frac{e^{-k\xi}}{2\pi
c}+\frac{e^{-\alpha\xi}}{2\pi}\int_{\BR}e^{i\xi
s}R(\lambda^{s}_{+}-\ep+is)ds.
\end{eqnarray*}
The absolute convergence in last integral yields
\[
|G_0(\xi)|\leq Ce^{-\alpha\xi},\,\,\,\ \xi\geq 0
\]
for some positive constant $C$. In the same manner, we can infer
that
\[
|G_0(\xi)|\leq Ce^{-\alpha\xi},\,\,\,\ \xi\leq 0.
\]
It is evident that the same reasoning works for $d>0$ since
$\Delta^{-1}_{0}(z)=O(|\Im z|^{-2})$ in the strip $|\Re
z|\leq\alpha$. Thus (\ref{te9}) is completed. Furthermore,
$(1+|\eta|)\Delta_{0}^{-1}(i\eta)\in L^2$ implies that $G\in H^1$
and $\widehat{G'_{0}}=is\Delta^{-1}_{0}(is)$. In addition, it
follows the same lines that
\begin{equation*}
|G'_{0}(\xi)|\leq Ce^{-\alpha|\xi|},\,\ \xi\in\mathbb{R}.
\end{equation*}
We now solve the inhomogeneous problem
\begin{equation}\label{teza2}
L_{0}v=h,\,\,\,\ h\in L^p
\end{equation}
for $h\in L^p$. For given $h\in L^p$, let $v$ be the convolution
$v=G_{+}\ast h$. Then by Young's inequality,
\[
||v||_{L^{p}}\leq ||G_{+}||_{L^{1}}||h||_{L^p}.
\]
Also note that
\[
||v'||_{L^{p}}\leq ||G'_{+}||_{L^{1}}||h||_{L^p}
\]
provided $d\neq 0$. To verify $v$ is a solution to (\ref{teza2}), it
is sufficient to show that (\ref{teza2}) holds everywhere for the
function $v$, that is,
\[
\int_{\BR}\chi(\xi)(L_{0}v)(\xi)d\xi=d\int_{\BR}\chi'(\xi)v'(\xi)d\xi-c\int_{\BR}\chi'(\xi)v(\xi)d\xi
+\int_{\BR}\chi(\xi)h(\xi)d\xi
\]
for all $C^\infty$ functions $\chi:\BR\rightarrow\mathbb{C}$ of
compact support. Indeed, it follows from the jump condition
(\ref{tezj2}) and Fubini's theorem that
\begin{eqnarray*}
\int_{\BR}\chi(\xi)(L_{0}v)(\xi)d\xi&=&a\int_{\BR}\chi(\xi)(G_{0}\ast
h)(\xi)d\xi+b\int_{\BR}\chi(\xi)(J\ast
G_{0}\ast h)(\xi)d\xi\\
&=&\int_{\BR}\chi(\xi)[b(J\ast G_{0})+aG_{0})\ast h](\xi)d\xi\\
&=&\int_{\BR}[\int_{\BR}\chi(\xi)(cG'_{0}(\xi-\eta)-dG''_{0}(\xi-\eta))d\xi]h(\eta)d\eta
+\int_{\BR}\chi(\eta)h(\eta)d\eta\\
&=&\int_{\BR}[\int_{\BR}(-c\chi'(\xi)G_{0}(\xi-\eta)+d\chi'(\xi)G'_{0}(\xi-\eta))d\xi]h(\eta)d\eta
+\int_{\BR}\chi(\eta)h(\eta)d\eta\\
&=&d\int_{\BR}\chi'(\xi)v'(\xi)d\xi-c\int_{\BR}\chi'(\xi)v(\xi)d\xi+\int_{\BR}\chi(\xi)h(\xi)d\xi.
\end{eqnarray*}
Now, to complete the proof, we only need to show that $L_0u=0$ for
some $u\in D(L_0)$ if and only if $u=0$. In fact, by interpreting
$u$ as tempered distribution and taking the Fourier transform, we
have
\[
\Delta_{0}(i\eta)\widehat{u}(\eta)=0.
\]
Since $\Delta_{0}(i\eta)\neq0$ for any $\eta\in\BR$, $\widehat{u}$
must be a zero distribution and hence $u=0$. The proof is completed
\end{proof}

We now construct the Green's function for a small perturbation
$L_{q}$ of $L_0$. Namely $L_{q}=L_{0}+Q$, here $Q:L^p\rightarrow
L^p$ is a bounded linear operator defined by
\[
(Qv)(\xi)=m(\xi)v(\xi)+n(\xi)\int_{\BR}J(\xi-\eta)v(\eta)d\eta,
\]
where $m,n\in L^{\infty}$.

\begin{proposition}\label{pte2}
Let $L_0$ be given by (\ref{te8}) and $L_{q}v(\xi)=(L_0+Q)v(\xi)$.
If
\[
\max\{||m||_{L^{\infty}},||n||_{L^{\infty}}\}\leq\varepsilon
\]
such that $\varepsilon$ is sufficiently small, then
$L_q:D(L_0)\rightarrow L^{p}$ is an isomorphism for $1\leq
p\leq\infty$. In addition, there exist positive constants $\nu$,$K$,
and the function $G_{q}:\BR^2\rightarrow\mathbb{C}$ satisfying the
pointwise estimates
\begin{equation}\label{teaa1}
|G_{q}(\xi,\eta)|\leq Ke^{\nu(\xi-\eta)}
\end{equation}
such that
\[
L^{-1}_{q}=\int_{\mathbb R}G_{q}(\xi,\eta)h(\eta)d\eta
\]
for each $h\in L^{p}$.
\end{proposition}
\begin{proof}
The proof is similar to the proof of . We shall only sketch the
proof of Proposition. Since $L_0:D(L_0)\rightarrow L^{p}$ is an
isomorphism, $L_q=(I-QL^{-1}_{0})L_0.$ It then follows from Theorem
1.16 of \cite{Kato} that
\begin{equation}\label{tezbb}
L^{-1}_{q}=L^{-1}_{0}\sum^{\infty}_{j=0}(QL^{-1}_{0})^{j}
\end{equation}
as long as $||QL^{-1}_{0}||<1.$ Set
\[
\Gamma_1(\xi,\eta)=m(\xi)G_0(\xi-\eta)+n(\xi)\int_{\BR}J(\xi-s)G_{0}(s,\eta)ds.
\]
In view of Lemma \ref{tte1}, $(QL^{-1}_{0})^j$ is an integral
operator, whose kernel is defined inductively by
\[
\Gamma_j(\xi,\eta)=\int_{\mathbb{R}}\Gamma_1(\xi,\tau)\Gamma_{j-1}(\tau,\eta)d\tau
\]
for all $j\geq 2$. Thanks to (H1), a straightforward calculation
shows that
\begin{equation}\label{te11}
|\Gamma_1(\xi,\eta)|\leq \varepsilon
Ce^{-\alpha|\xi-\eta|}+\varepsilon
e^{-\alpha|\xi-\eta|}\int_{\mathbb{R}}J(s)e^{\alpha|s|}ds.
\end{equation}
Therefore, there exists positive constant $K_1$ such that
$|\Gamma_1(\xi,\eta)|\leq 2\varepsilon Ke^{-\alpha|\xi-\eta|}.$ In
addition, by using (\ref{te11}), we infer that
\[
|\Gamma_j(\xi,\eta)|\leq {\Psi}^{\ast j}(\xi-\eta),\,\,
\Psi(\xi)=2\varepsilon K_1e^{-\nu|\xi|},
\]
where $\Psi^{*j}=\Psi*\Psi^{*(j-1)}$, is the $j$-fold convolution of
$\Psi$ with itself. By Lemma 5.1 of \cite{Mall}, we infer that
\begin{equation}\label{ptzz0}
\sum^{\infty}_{j=1}|\Gamma_{k}(\xi,\eta)|\leq K_2e^{-\nu|\xi-\eta|}.
\end{equation}
Here $\nu=\sqrt{\alpha^2-4\varepsilon K\alpha}$ and
$K_2=\frac{2\varepsilon K\alpha}{\beta}$.  Now Let
\[
G_{q}(\xi,\eta)=G_0(\xi-\eta)+\int_{\mathbb
R}G_0(\xi-s)(\sum_{j=1}^{\infty}\Gamma_j(s,\eta))ds.
\]
Then a direct calculation yields
\[
|G_{q}(\xi,\eta)|\leq Ke^{-\nu|\xi-\eta|}.
\]
Furthermore, it is easy to see that
\[
L^{-1}_{q}h=\int_{\mathbb R}G_{q}(\xi,\eta)h(\eta)d\eta.
\]
Therefore, the proof is completed.

\end{proof}

Next we consider the operator $L$ defined by (\ref{teb3}).
Hereafter, we assume that
\begin{equation}\label{tca1}
\lim_{\xi\rightarrow+\infty}a(\pm\xi)=a^{\pm},\,\,
\lim_{\xi\rightarrow+\infty}b(\pm\xi)=b^{\pm},
\end{equation}
where $a^{+}, a^{-}, b^{+},$ and $b^{-}$ are constants. Let
$L_{\pm}:D(L_0)\rightarrow L^{p}$ be the  operator defined by
\[
L_{\pm}u=du''-cu'+a^{\pm}+b^{\pm}J*u,
\]
respectively.
\begin{definition}
The operator $L$ is called asymptotic hyperbolic
if both $L_{+}$ and $L_{-}$ are hyperbolic.
\end{definition}

\begin{proposition}\label{tezpp1}
Assume that \ref{tca1} is satisfied and $L$ is asymptotically
hyperbolic. Suppose that there are bounded sequences $u_n\in D(L)$
and $h_n\in L^p$ such that $Lu_n=h_n$ and $h_n\rightarrow h^{*}$ in
$L^{p}$. Then there exists a subsequence $u_{n^{\prime }}$ and some
$u^{*}\in D(L)$ such that $u_{n^{\prime }}\rightarrow u^{*}$ in
$D(L)$ and $Lu^{*}=h^{*}.$ The same conclusion hold true for
$L^{*}.$
\end{proposition}
\begin{proof}
Due to the assumption, for any $\varepsilon>0$, there exists
$\tau(\varepsilon)>0$ such that $|a(\pm\xi)-a^{\pm}|\leq\varepsilon$
and $|b(\pm\xi)-b^{\pm}|\leq\varepsilon$ whenever $\xi\geq\tau$. Now
let
\[
\theta _{\tau}(\xi)=\left\{
\begin{array}{ll}
1,\,\,\ & \xi \geq \tau,  \\
0,\,\,\ & \xi <\tau,
\end{array}
\right.
\]
Let the bounded linear operators $Q^{\pm}: L^{p}\rightarrow L^{p},$
defined by
\[(Q^{\pm}v)(\xi)=\theta _{\tau}(\pm\xi)a(\xi)+\theta
_{\tau}(\xi)b(\xi)\int_{\mathbb{R}}J(\xi-\eta)v(\eta)d\eta,
\] respectively. Set $L_{q^{\pm}}=L_{\pm}+Q^{\pm}$. It is clear that
$L_{q^{\pm}}u_{n}=(L_{q^{\pm}}-L)u_{n}+f_n$.  $p>1,$then the
embedding theorem implies that the sequence $u_n$ is equicontinuous
on any compact interval. In case of $d>0,$ $u_n^{\prime }$ is also
equicontinuous on any compact interval. When $p=1$, by means of an
argument similar to one used in \cite{Mall}, it can be shown that
the above conclusions are still true. Therefore there is a
subsequence, still labeled by $u_n$, which converges to $u^{*}$
uniformly on any compact interval for some function
$u^{*}:\BR\rightarrow \mathbb C$. Clearly, $u^{*}\in L^\infty $
since $u_n$ is bounded in $W^{1,p}.$

Next we show that $u_n\rightarrow u^{*}$ in $L^p$ if $1\leq p<\infty
.$ By the assumption, there exists $\widehat{C}>0$ such that
$||u_n||_{L^\infty }\leq \widehat{C}$. Let
\[
g_n(\xi )=K_2\int_{\BR}e^{-\mu |\xi -\eta |}|h_n(\eta)|d\eta ,\quad
\quad g^{*}(\xi )=K_2\int_{\BR}e^{-\mu |\xi -\eta
|}|h^{*}(\eta)|d\eta ,
\]
then it follows from (\ref{te23}) that
\[
|u_n(\xi )|\leq \widehat{C}K_1e^{-\mu |\xi |}+g_n(\xi),\quad \quad
\xi \in \BR.
\]
By using the H\"{o}lder inequality and Young inequality, we find
that $g_n,$ $g^{*}\in L^{p}\cap L^{\infty}$, and
\[
\lim_{n\rightarrow\infty}||g_n-g^{*}||_{L^{\infty}}=0, \qquad
\lim_{n\rightarrow\infty}||g_n-g^{*}||_{L^{p}}=0.
\]
The generalized Lebesgue dominated convergence theorem implies
$||u_n||_{L^{p}}\rightarrow ||u^{*}||_{L^p}$. Finally,
Br\'{e}zis-Lieb Lemma yields
\[
\lim_{n\rightarrow\infty}||u_n-u^{*}||_{L^{p}}=0.
\]
Due to (\ref{te24}), we have
\[
u_n=\Lambda _\omega ^{-1}((\lambda-\omega)Iu_n-Lu_n+h_n).
\]
By passing the limit $n\rightarrow \infty $, we see that
\[
u^{*}=\Lambda _\omega ^{-1}((\lambda-\omega)Iu^{*}-Lu^{*}+h^{*}).
\]
Namely,
\[
(\Pi_L-\lambda I)u^{*}=h^{*}.
\]
Furthermore, applying (\ref{tezap001}) or (\ref{tezap002}) to the
difference $u_n-u^{*}$ yields
\[
u_n\rightarrow u^{*}\,\ \text{in}\,\, W^{1,p} (\text{or}\,\
W^{2,p}\,\ \text{if}\,\ d>0).
\]
It remains to show that the assertion is valid when $p=\infty$. We
first write for each $u_n$ in the form
\[
cu_n(\xi _1)=cu_n(\xi _2)+\int_{\xi _2}^{\xi _1}[L(\eta )u_n(\eta
)-h_n(\eta )d\eta,\,\,\ \text{for}\,\ d=0
\]
and
\[
du_n^{\prime }(\xi _1)=du_n^{\prime }(\xi _2)+\int_{\xi _2}^{\xi
_1}[cu_n^{\prime }(\eta )-Lu(\eta )+h_n(\eta )]d\eta,\,\,\
\text{for}\,\ d>0,
\]
where $-\infty <\xi _2<\xi _1<\infty $. Since $(J*u_n)(\cdot)$
converges $(J*u^{*})(\cdot)$ pointwise and $J*u_n$ is uniformly
bounded, upon taking the limit, we find
\[
cu^{*}(\xi _1)=cu^{*}(\xi _2)+\int_{\xi _2}^{\xi _1}[L(\eta
)u^{*}(\eta)-h_n(\eta )d\eta,\,\,\ \text{for}\,\ d=0
\]
and
\[
d{u^{*}}'(\xi _1)=d{u^{*}}'(\xi _2)+\int_{\xi _2}^{\xi
_1}[c{u^{*}}'(\eta )-Lu(\eta )+h_n(\eta )]d\eta,\,\,\ \text{for}\,\
d>0.
\]
Therefore, we have
\[
d(u^{*})''(\xi)-c(u^{*})'(\xi)+(Lu)(\xi)=h(\xi),\,\,\ \text{for
any}\,\ \xi\in\mathbb{R}.
\]
Applying (\ref{te23}) to $u_n-u^{*}$ yields that
\[
|(u_n-u^{*})(\xi)|\leq 2\widehat{C}K_1e^{-\mu
|\xi|}d\xi+K_2\mu||h_n-h^{*}||_{L^{\infty}}.
\]
Since $h_n\rightarrow h^{*}$ in $L^{\infty}$, for any
$\varepsilon>0$, there exist positive constants $N(\varepsilon)$ and
$T(\varepsilon)$ such that
$|(u_n-u^{*})(\xi)|\leq\frac{1}{2}\varepsilon$ whenever
$n>N(\varepsilon)$ and $|\xi|>T(\varepsilon)$. In addition, we
already know that $u_n$ uniformly converge $u^{*}$ on any compact
interval. Hence, there exists $\widetilde{N}(\varepsilon)>0$ such
that $||(u_n-u^{*})||_{L^{\infty}}\leq\varepsilon$ if
$n>\widetilde{N}(\varepsilon)$. Once again, the similar reasoning
shows that $u_n\rightarrow u^{*}$ in $W^{1,\infty}$ (or
$W^{2,\infty}$). Thus, the proof is completed.
\end{proof}

\begin{proposition}\label{pte3}
Let $(c,U)$ be the solution to (\ref{te1}), then there exist
positive constants $\nu $ and $C_\nu $ such that
\begin{equation}\label{te14}
|U^{\prime }(\xi)|\leq C_\nu e^{(\lambda^{s} _{+}+\nu )\xi},\,\,\
\xi\geq 0
\end{equation}
and
\begin{equation}\label{te18}
|U^{\prime}(\xi)|\leq C_\nu e^{(\lambda^{u} _{-}-\nu)\xi},\,\,\
\xi\leq 0,
\end{equation}
where $\nu<\min\{-\lambda^{s} _{+},\lambda^{u}_{-}\}$.
\end{proposition}
\begin{proof}
We shall use Proposition \ref{pte2} to derive (\ref{te14}) and
(\ref{te18}). Recall that
\[
M_{\pm}(\xi)v=[f_r(U,J\ast U)-f_r(\pm 1,\pm 1)]v+[f_s(U,J\ast
U)-f_s(\pm 1,\pm 1)]J\ast v.
\]
Due to (\ref{teb2}), for any $\varepsilon>0$, there exists $\tau>0$
such that
\[
||M_{+}(\xi)||\leq\varepsilon,\,\ \text{as}\,\
\xi\geq\tau,\,\,\text{and}\,\, ||M_{-}(\xi)||\leq\varepsilon,\,\
\text{as}\,\ \xi\leq-\tau.
\]

Now let
\begin{equation}\label{te2.2}
\vartheta _{+}(\xi)=\left\{
\begin{array}{ll}
1,\,\,\ & \xi \geq \tau,  \\
0,\,\,\ & \xi <\tau,
\end{array}
\right. \quad \quad \vartheta _{-}(\xi )=\left\{
\begin{array}{ll}
0,\,\,\ & \xi >-\tau,  \\
1,\,\,\ & \xi \leq -\tau.
\end{array}
\right.
\end{equation}
We also set
\[
\overline{L}_{\pm}(\xi
)=L_{\pm}+\vartheta_{\pm}(\xi)M_{\pm}(\xi),\qquad
\overline{M}_{\pm}(\xi)=(1-\vartheta_{\pm}(\xi))M_{\pm}(\xi).
\]

Let $V$ stand for $U^{\prime}$ and Let
\[
(\Pi_{\overline{L}_{\pm}}v)(\xi)=dv''(\xi)-cv'(\xi)+\overline{L}_{\pm}v(\xi).
\]
Then $(\Pi_{\overline{L}_{\pm}}V)(\xi)=-\overline{M}_{\pm}(\xi
)V(\xi).$ Since $\overline{M}_{\pm}(\xi)$ is a bounded linear
operator and $\varepsilon$ can be made arbitrary small by
manipulating $\tau$, we can choose $\tau$ sufficiently large such
that the operators $\Pi_{\overline{L}_{\pm}}$ satisfies the
conditions of Proposition \ref{pte2}. Let $\overline{G}_+:{\mathbb
R}^2\rightarrow \mathbb C$ denote the Green's function for $\Pi
_{\overline{L}_+}$, which enjoys the estimate (\ref{teaa1}).
Therefore, for every $\xi \in\mathbb R$,
\begin{eqnarray*}
V(\xi) &=&\int_{-\infty }^\infty \overline{G}_+(\xi ,\eta
)[-\overline{M}_+(\eta )V(\eta
)]d\eta \\
&=&\int_{-\infty }^\tau \overline{G}_+(\xi,\eta)[-M_1(\eta)V(\eta
)]d\eta ,
\end{eqnarray*}
where we use the fact that $M_1(\eta)=0$ for all $\eta\geq\tau $.
Consequently, for any $\xi \geq \tau $, we have
\begin{eqnarray*}
|V(\xi)| &\leq &C_1\int_{-\infty }^\tau e^{(\lambda^{s}_{+}+\nu
)(\xi
-\eta )}||M||_{L^\infty }||V||_{L^\infty }d\eta \\
&\leq &\widetilde{C}_1 e^{(\lambda^{s}_{+}+\nu )\xi }.
\end{eqnarray*}
Since $V$ is bounded on $\mathbb R$, it is possible to choose
$C_\nu>0$ such that the desired estimate (\ref{te14}) holds for all
$\xi\geq 0$. Analogously,
\[
|U^{\prime}(\xi)|\leq C_\nu e^{(\lambda^{u} _{+}-\nu)\xi},\,\,\
\xi\leq 0.
\]
The proof is completed.
\end{proof}
Now we are ready to give the main result in this section

\begin{theorem}\label{ttz3}
Let $(c,U)$ be the solution to (\ref{te1}), then there exist
positive constants $D_1$ and $D_2$ such that
\begin{equation}\label{tezmt1}
U(\xi)=1-D_1e^{\lambda^{s}_{+}\xi }[1+o(1)],\,\,\ \text{as} \,\
\xi\rightarrow \infty,
\end{equation}
and
\begin{equation}\label{tezmt2}
U(\xi)=-1+D_2e^{\lambda^{u}_{-}\xi }[1+o(1)],\,\,\ \text{as} \,\
\xi\rightarrow -\infty,
\end{equation}
where $\lambda^{s}_{+}<0, \lambda^{u}_{-}>0$ are the roots of
$\Delta_{L_\pm}(z),$ respectively.
\end{theorem}
\begin{proof}  We first show that
\begin{equation}\label{tezha1}
|M_{+}(\xi)V(\xi)|\leq C_2e^{2(\lambda^{s}_{+}+\nu )\xi },\qquad \xi
\geq 0,
\end{equation}
\begin{equation}
|M_{-}(\xi)V(\xi)|\leq C_2e^{2(\lambda^{u}_{-}-\nu )\xi },\qquad \xi
\leq 0
\end{equation}
hold true for some constant $C_2>0$. In fact, by mean value theorem,
we have
\[
|M_{\pm}(\xi)V(\xi)|\leq
\widetilde{K}[|U(\xi)\mp1|+|J*(U\mp1)(\xi)|][|V(\xi)|+|J*V(\xi)|]
\]
for some positive constant $\widetilde{K}$.  It follows from
proposition \ref{pte3} that
\[
|U(\xi)-1|\leq (\lambda^{s}_{+}+\nu )^{-1}C_\nu
e^{(\lambda^{s}_{+}+\nu )\xi },\,\,\ \xi \geq 0\,\,\ \text{and}\,\
|U(\xi)+1|\leq (\lambda^{u} _{-}-\nu )^{-1}C_\nu e^{(\lambda^{u}
_{-}-\nu)\xi},\,\,\ \xi \leq 0.
\]
Hence, for any $\xi\geq 0$,
\begin{eqnarray*}
&&|J*(U-1)(\xi)| \\
&=&|\int_{\mathbb R}J(\eta )(U(\xi -\eta )-1)d\eta| \leq
\int_{-\infty }^\xi +\int_\xi ^\infty J(\eta
)|U(\xi -\eta )-1|d\eta \\
&\leq &\frac{C_\nu e^{(\lambda^{s}_{+}+\nu )\xi }}{(\lambda^{s}_{+}+\nu )}%
\int_{\mathbb R}J(\eta )e^{-(\lambda^{s}_{+}+\nu )\eta }d\eta
+e^{(\lambda^{s}_{+}+\nu)\xi }\int_{\mathbb R}J(\eta
)e^{-(\lambda^{s}_{+}+\nu )\eta }|U(\xi -\eta)-1|d\eta\\
&\leq& C'e^{(\lambda^{s}_{+}+\nu)\xi}.
\end{eqnarray*}
Similarly,
\[
|J*(U+1)(\xi)|\leq Ce^{(\lambda^{u}_{-}-\nu)\xi},\,\,\, \xi\leq 0,
\]
and
\[
|J*V(\xi)|\leq\left\{\begin{array}{ll}
     C_3 e^{(\lambda^{s}_{+}+\nu)\xi}, &\mbox{if $\xi\geq 0,$}\\
     C_3 e^{(\lambda^{u} _{-}-\nu)\xi}, &\mbox{if $\xi\leq 0,$}\end{array}\right.
\]
Therefore,(\ref{tezha1}) follows.

Now set $h_{\pm}(\xi)=-M_{\pm}(\xi)U(\xi)$. As long as $\nu $ is
sufficiently small, there exists $\iota >0$ such that
$2(\lambda^{s}_{+}+\nu )\leq \lambda^{s}_{+}-\iota$ and
$2(\lambda^{u} _{-}-\nu )\geq \lambda^{u}_{-}+\iota $. Therefore, we
have
\begin{equation}\label{atz1}
|h_+(\xi)|\leq C_2e^{(\lambda^{s}_{+}-\iota)\xi},\,\,\ \text{for
any}\,\ \xi\geq0,\,\,\ |h_-(\xi)|\leq C_2e^{(\lambda^{u}
_{-}+\iota)\xi},\,\,\ \text{for any}\,\xi\leq 0.
\end{equation}
In addition, due to the boundedness of $U\mp1$ and $J*(U\mp1)$, it
is easy to see that
\begin{equation}\label{atzsu1}
|h_+(\xi)|=O(e^{(\lambda^{u}_{-}-\nu)\xi}),\,\,\ \text{as}\,\
\xi\rightarrow-\infty,\,\,\ |h_-(\xi)|=O(e^{(\lambda^{s}
_{+}+\nu)\xi}),\,\,\ \text{as}\,\ \xi\rightarrow\infty.
\end{equation}

Clearly, we have
\begin{equation}
dV''-cV^{\prime}+L_{\pm}V=h_{\pm}(\xi).
\end{equation}
In particular, when $d=0$,
\begin{equation}\label{tezbb2}
|V'(\xi)|\leq|b^\pm||J*V(\xi)|+|a^\pm||V(\xi)|+|h_\pm(\xi)|.
\end{equation}
We also observe that $h_{\pm}$ is differentiable and
\[
|h'_{\pm}(\xi)|\leq
K'[|V(\xi)|^2+|V(\xi)||J*V(\xi)|+|U(\xi)\mp1||V'(\xi)|+|U(\xi)\mp1||J*V'(\xi)|].
\]
Therefore, it follows from (\ref{tezbb2}) that
\begin{equation}\label{tezbb3}
|h'_+(\xi)|\leq C_4e^{(\lambda^{s}_{+}-\iota)\xi},\,\,\ \text{for
any}\,\ \xi\geq0,\,\,\ |h'_-(\xi)|\leq C_4e^{(\lambda^{u}
_{-}+\iota)\xi},\,\,\ \text{for any}\,\xi\leq 0.
\end{equation}

Next, we show (\ref{tezmt1}). Thanks to (\ref{atz1}) and
(\ref{atzsu1}), $\widehat{h}_{+}(z)$ is analytic in the strip
$0\leq\text{Im}z\leq2\epsilon-\lambda^{s}_{+}$, where
$0<2\epsilon<\iota$ and
$\widehat{g}(z)=\frac{1}{2\pi}\int_{\BR}e^{-izs}g(s)ds$. In case
$d=0$, let $h^{\rho}_+(\xi)=h_+(\xi)e^{\rho\xi},$ where
$\rho\in(0,2\epsilon-\lambda^{s}_{+})$. Due to (\ref{tezbb3}), for
each $\rho\in(0,2\epsilon-\lambda^{s}_{+})$, $h^{\rho}_{+}\in
W^{1,p}(\mathbb R)$ for any $p\geq1$. Furthermore, we have
\[
|\eta||\widehat{h}_+(\eta+i\rho )|=|\eta
||\widehat{h^{\rho}_+}(\eta)|=|i\eta \widehat{h^{\rho}_+}(\eta
)|=|\widehat{\partial_{\eta}h^{\rho}_+}(\eta)|\leq
||\partial_{\eta}h^{\rho}_+ ||_{L^1},\quad \eta \in \BR.
\]
Therefore, in the strip $0\leq\text{Im}z\leq
2\epsilon-\lambda^{s}_{+},$
\[
|\widehat{h}_+(z)|=O(|\text{Re}z|^{-1}),\quad
|\text{Re}z|\rightarrow \infty.
\]
In the strip $|\text{Re}z|\leq D $ with any fixed $D>0$,  we have
that $\Delta_{L_+}(z)=O(|\text{Im}z|)$ ( $=O(|\text{Im}z|^2)$ if
$d>0$) uniformly  as $|\text{Im}z|\rightarrow \infty $ .
Consequently,
$\widehat{h}_{+}(-iz)\Delta^{-1}_{L_+}(z)=O(|\text{Im}z|)^{-2}$ for
any $z\in\mathbb{C}$ with
$0<\text{Re}z\leq\lambda^{s}_{+}-2\epsilon$ and $\Delta_{L_+}(z)\neq
0.$

Since $\Pi _{L_+}$ is an isomorphism, $V$ is the unique solution to
$dv''-cv'+L_{+}v=h_{+}$. By using Fourier transform and shifting the
integrating path, when $\xi\geq 0,$ we find
\begin{eqnarray*}
V&=&\frac{1}{2\pi}\int_{\mathbb
R}\frac{e^{i\eta\xi}\widehat{h}_{+}(\eta)}{\Delta_{+} (i\eta)}d\eta
=\frac{-i}{2\pi}\int_{\mathbb
R}\frac{e^{i\eta\xi}\widehat{h}_{+}(-i(i\eta))}{\Delta_{+}
(i\eta)}d(i\eta)\\
&=&\sum\text{Res}\frac{e^{z\xi}\widehat{h}_{+}(-iz)}{\Delta_{+}
(z)}|_{\lambda^{s}_{+}-\epsilon\leq\text{Re}z\leq
0}+\frac{-i}{2\pi}\int_{\text{Re}z= \lambda^{s}
_{+}-\epsilon}\frac{e^{z\xi
}\widehat{h}_{+}(-iz)}{\Delta_{L_+}(z)}dz.
\end{eqnarray*}
Here we choose $\epsilon$ such that $\Delta_{L_+}(\lambda^{s}
_{+}-\epsilon+i\eta)\neq 0$ for any $\eta\in\mathbb{R}.$ Clearly,
The last integral absolutely converges.

Let $\Upsilon _{\lambda _{+}^s-\epsilon }=\{z\in \mathbb
C|\Delta_{L^{+}}(z)=0,\lambda _{+}^s-\epsilon <\text{Re}z<0\}$.
Since $\widehat{h}(z)$ is analytic in the strip $0\leq
\text{Im}z\leq 2\epsilon -\lambda _{+}^s $, in the strip $\lambda
_{+}^s-2\epsilon\leq \text{Re}z\leq 0,$ $\widehat{h}(-iz)\Delta
_{L_{+}}^{-1}(z)$ is meromorphic and only has poles which may occur
at $z\in \Upsilon _{\lambda _{+}^s-\epsilon }$. We claim that
$\widehat{h}(iz)\Delta _{L_{+}}^{-1}(z)$ has a simple pole at
$z=\lambda _{+}^s$. Suppose this is not true. In virtue of Lemma
\ref{tep1}, in the strip $\lambda _{+}^s-\epsilon <\text{Re}z< 0$,
either all the poles of $\widehat{h}(-iz)\Delta _{L_{+}}^{-1}(z)$
occur at $z\in \Upsilon _{\lambda _{+}^s-\epsilon }$ with
$\text{Re}z<\lambda _{+}^s$  or $\widehat{h}(-iz)\Delta
_{L_{+}}^{-1}(z)$ is analytic . For the latter,
$V(\xi)=O(e^{(\lambda _{+}^s-\epsilon )\xi }),$ as $\xi \rightarrow
\infty $. Certainly, $h_{+}(\xi)=O(e^{2(\lambda _{+}^s-\epsilon )\xi
})$ and $\widehat{h}(iz)$ is analytic in the strip $2\lambda
_{+}^s-\varkappa $ for some $0<\varkappa \leq 2\epsilon.$ Hence the
path of integration can be shifted to the line $\text{Re}z=2\lambda
_{+}^s-\varkappa.$ Consequently, one of the following cases must
occur.

Case I The set $Z$ is not empty, where
\[
Z=\{\text{Re}z\in \mathbb{R^{-}}|\Delta_{L_+}(z)=0,\
\widehat{h}(-iz)\,\ \text{is analytic at}\,\ z\,\text{and}\,\
\widehat{h}(-iz)\Delta_{L_{+}}^{-1}(z)\,\, \text{has poles at}\,\
z\}.
\]

Case II   $V(\xi )=O(e^{-b\xi})$ for any $b\in \BR^{+}$.

Next we show that both case I and II are impossible. We start with
case (I). Let $\varrho=\sup Z$.  Recall that $\lambda _{+}^s$ is the
only real zero of $\Delta _{L_{+}}$ in the half plane
$\text{Re}z\leq 0$. By lemma\ref{tep1}, we may assume
$\varrho\pm\mu_{m}i$ with $\mu_m>0 (1\leq m<\infty)$ are all the
element of $Z$ with real part equal to $\varrho$. Suppose that
$\widehat{h}(-iz)\Delta_{L_{+}}^{-1}(z)$ has a pole of order $l_m+1$
at $\varrho+\mu_{m}i$. Then
\begin{eqnarray*}
V(\xi) &=&\sum\text{Res} (e^{\xi z}\widehat{h}(iz)\Delta
_{L_{+}}^{-1}(z))_{\text{Re}z=\varrho}+
o(e^{\varrho\xi}) \\
&=&\sum_{m}
p_{l_m}(\xi)e^{\varrho\xi}\cos(\mu_m\xi+k_m)+o(e^{\varrho\xi}),
\end{eqnarray*}
where $p_{l_m}$ are real polynomials and $k_m\in\BR$. Thus,
$V(\xi)=\xi^{N}e^{\varrho\xi}(q(\xi)+O(\xi^{-1}))$ as
$\xi\rightarrow\infty$ for some $N>0$, where $q$ is a quasiperiodic
function of mean value zero. According to \cite{Mall1},
\[
\liminf_{\xi\rightarrow\infty}q(\xi)<0<\limsup_{\xi\rightarrow\infty}q(\xi).
\]
Consequently, $V(\xi_1)<0$ for some $\xi_1>0$. This contradicts the
fact that $V(\xi)>0$ for any $\xi\in(-\infty ,\infty).$ Therefore
case(I) never occurs.

For the case II, we define
\[
V_b=\int_\BR V(\xi )e^{b\xi },\quad b\in\BR^{+}.
\]
Note that
\[
\int_{\BR} e^{b\xi}J*V(\xi)d\xi =\int_{\BR} e^{b\eta }V(\eta
)\int_{\BR} e^{b(\xi -\eta )}J(\xi -\eta )d\xi d\eta =V_b\int_{\BR}
e^{b\xi }J(\xi )d\xi .
\]
Let
\[
\underline{a}=\min_{[-1,1]\times[-1,1]}f_r(r,s),\,\
\underline{b}=\min_{[-1,1]\times[-1,1]}f_s(r,s).
\]
Due to (H2) and (H3), $\underline{a}>-\infty$ and $\underline{b}>0$.
Consequently,
\begin{equation}\label{atz4}
cV'-dV''\geq \underline{a}V+\underline{b}J*V
\end{equation}
Multiplying each side of (\ref{atz4}) by $e^{b\xi}$ and integrating
by part yield
\[
(-cb-db^2-\int_{\BR}J(\xi)e^{b\xi}d\xi)V_b\geq\underline{a}V_b,
\]
thus
\[
-cb-db^2-\underline{b}\int_{\BR}J(\xi)e^{b\xi}d\xi\geq\underline{a}.
\]
Since $-cb-db^2-\int_{\BR}J(\xi)e^{b\xi}d\xi\rightarrow-\infty $, as
$b\rightarrow \infty ,$ we arrive at a contradiction. This implies
that Case (II) can not occur. Therefore,
$e^{z\xi}\widehat{h}(-iz)\Delta _{L_{+}}^{-1}(z)$ has a simple pole
at $z=\lambda _{+}^s$, and
\begin{eqnarray*}
V(\xi)&=&\frac{e^{\lambda^{s}_{+}\xi}\widehat{h}_{+}(-i\lambda^{s}_{+})}
{\Delta'_{L_+}(\lambda^{s}_{+})}
+\frac{1}{2\pi}\int_{\text{Re}z=\lambda^{s}_{+}-\epsilon}\frac{e^{z\xi
}\widehat{h}_{+}(-iz)}{\Delta_{L_+}(z)}dz\\
&=&\frac{e^{\lambda^{s}_{+}\xi}\int_{\BR}h(\eta)e^{-\lambda^{s}_{+}\eta}d\eta}{
\int_{\BR}\eta J(\eta)e^{-\lambda^{s}_{+}\eta}d\eta-c}
+\frac{e^{(\lambda^{s}_{+}-\ep)\xi}}{2\pi}\int_{\BR}\frac{e^{i\xi
s}\widehat{h}_{+}(\eta+i(\ep-\lambda^{s}_{+}))}{\Delta_{L_+}(\lambda^{s}_{+}-\ep+i\eta)}d\eta.
\end{eqnarray*}
Now let
\[
\gamma^{+} =\frac{\int_{\BR}h(\eta)e^{-\lambda^{s}_{+}\eta}d\eta}
{\int_{\BR}\eta J(\eta)e^{-\lambda^{s}_{+}\eta}d\eta-c} ,\,\,\,\
V^{+}(\xi)=\frac{e^{(\lambda^{s}_{+}-\ep)\xi}}{2\pi}\int_{\BR}\frac{e^{i\xi
s}\widehat{h}_{+}(\eta+i(\ep-\lambda^{s}_{+}))}{\Delta_{L_+}(\lambda^{s}_{+}-\ep+i\eta)}d\eta.
\]
Clearly, $V^{+}(\xi)=o(e^{\lambda^{s}_{+}\xi}),\,\, \text{as}\,\
\xi\rightarrow +\infty$. The positivity of $V$ forces that
$\gamma^{+}>0$. Thus,
\[
U'(\xi)=\gamma^{+}e^{\lambda^{s}_{+}\xi}+o(e^{\lambda^{s}_{+}\xi}),\,\,\
\xi\rightarrow +\infty.
\]
By considering the equation $cV'-dV''=L_{-}V+h_{-}$ and arguing
analogously, we may find
\[
U'(\xi)=\gamma^{-}e^{\lambda^u_{-}\xi}+o(e^{\lambda^u_{-}\xi}),\,\,\
\xi\rightarrow -\infty
\]
for some constant $\gamma^{-}>0$. With the boundary conditions
$U(\pm\infty)=\pm 1$, we are readily to obtain the desired
conclusions.
\end{proof}

The uniqueness of monotone traveling wave $U$ for nonlocal
Allen-Cahn equation (\ref{te1}) with $d=0$ were established in
\cite{bates3} and late for the general equation (\ref{te1}) in
\cite{chen1}. In those works, the uniqueness of speed and profile of
traveling wave solution $U$ are obtained by means of a comparsion
principle and sub- and super solution techniques. Here we provided a
technically different and simplified proof for the uniqueness of
$U$.
\begin{corollary}
Assume that (H1)-(H5) are satisfied. Then there exists a unique
$c^{*}\in\BR$ such that equation (\ref{te1}) possesses a solution
satisfying (\ref{te3}) if and only if $c=c^{*}$ and the traveling
wave solution $U$ is unique up to translation of $\xi.$
\end{corollary}

\begin{proof} We shall argue by contradiction. Suppose that there exist
$(c_i,U_i)$ satisfying (\ref{te1}) with $c_1<c_2$ , $i=1,2.$ We may
assume that one of these solutions has speed $c^{*}$. By Theorem
\ref{ttz3}, both solutions satisfy

\[
U_i(\xi )=\left\{
\begin{array}{ll}
-1+n_{i}e^{\varrho^{u}_{i}\xi}+o(e^{\varrho^{u}_{i}\xi}), & \xi
\rightarrow -\infty ,
\\
1-\widetilde{n}_{i}e^{\varrho^{s}_{i}\xi}+o(e^{\varrho^{s}_{i}\xi}),
& \xi \rightarrow \infty
\end{array}
\right.
\]
for some $n_i,\widetilde{n}_i>0$ and
$\varrho^{u}_{i}>0,\varrho^{s}_{i}<0$. Furthermore,
$\varrho^{u}_{i}$ and $\varrho^{s}_{i}$ satisfy
\begin{eqnarray*}
c_i\varrho^{s}_{i}-d(\varrho^{s}_{i})^2-a^{+}-b^{+}\int_{\BR}J(s)e^{-\varrho^{s}_{i}s}ds &=&0, \\
c_i\varrho^{u}_{i}-d(\varrho^{u}_{i})^2-a^{-}-b^{-}\int_{\BR}J(s)e^{-\varrho^{u}_{i}s}ds
&=&0.
\end{eqnarray*}
In view of the proof of lemma \ref{tep1}, it is easy to see that
\[
\varrho^{s}_{1}<\varrho^{s}_{2}<0,\quad
0<\varrho^{u}_{1}<\varrho^{u}_{2}.
\]
Thus, $U_2(\xi )<U_1(\xi)$ for all sufficiently large $|\xi|.$ This
together with the monotonicity of $U_i$ justify that we can choose
$\tau\in\BR$ and replace $U_2(\xi )$ by $U_2(\xi +\tau )$ such that
$U_2(\xi)\leq U_1(\xi)$ for all $\xi \in\BR$ and $U_2(\xi
^{*})=U_1(\xi_{0})$ for some $\xi_{0}$. Consequently,
$U_2^{\prime}(\xi_{0})=U_1^{\prime}(\xi_{0})$ and $U_2^{\prime
\prime}(\xi_{0})\leq U_1^{\prime \prime }(\xi_{0}) $. Moreover, (H2)
and the fact that $J*U_1(\xi_{0})\geq J*U_2(\xi_{0})$ imply that
$f(U_1(\xi_{0}),J*U_1(\xi_{0}))\geq f(U_2(\xi_{0}),J*U_2(\xi
_{0})).$ By plugging these relations into (\ref{te1}), we find
\[
0=dU_1^{\prime \prime}(\xi_{0})-c_1U_1^{\prime }(\xi_{0})+f(U_1(\xi
^{*}),J*U_1(\xi_{0}))>dU_2^{\prime \prime}(\xi
_0)-c_2U_2^{\prime}(\xi_{0})+f(U_2(\xi_{0}),J*U_2(\xi_{0}))=0.
\]
The contradiction completes the proof.
\end{proof}

\section{Spectral analysis of traveling wave $U$}
In this section, we study the spectrum of the operator $\Pi_L$.
Recall that
\begin{equation}\label{tezspa1}
(\Pi_Lu)(\xi):=du''(\xi)-cu'(\xi)+f_r(U,J*U)u(\xi)+f_s(U,J*U)(J*u)(\xi)
\end{equation}
and
\begin{equation}
(\Pi_{L_\pm}u)(\xi):=du''(\xi)-cu'(\xi)+a^{\pm}u(\xi)+b^{\pm}(J*u)(\xi).
\end{equation}

Clearly, the equation $(\Pi_L-\lambda I)u=0$ is equivalent to
\begin{equation}\label{te20}
du''(\xi)-cu'(\xi)+L(\xi )u(\xi)-\lambda u(\xi)=0.
\end{equation}

Let
\[
L^{*}(\xi)v(\xi)=f_r(U(\xi),J*U(\xi))v(\xi)+\int_{\BR}J(\xi-\eta)f_s(U(\eta),J*U(\eta))v(\eta)d\eta.
\]
The adjoint equation of (\ref{te20}) is the equation
\begin{equation}
dv''(\xi)+cv'(\xi)+L(\xi)v(\xi)-\overline{\lambda }v(\xi)=0,
\end{equation}
where $\overline{\lambda }$ denotes the conjugate of $\lambda.$ We
define the formally adjoint operator $(\Pi_L-\lambda I)^{*}$ of
$(\Pi_L-\lambda I)$ to be
\[
((\Pi_L-\lambda I)^{*}v)(\xi )=dv''(\xi)+cv'(\xi)+L^{*}(\xi)v(\xi
)-\overline{\lambda}v(\xi).
\]
It is easy to see that
\begin{equation}\label{tezpe}
\int_{\mathbb R}\overline{v(\xi)}((\Pi_L-\lambda I)u)(\xi)d\xi
=\int_{\mathbb R}\overline{((\Pi_L-\lambda I)^{*}v)(\xi)}u(\xi)d\xi,
\end{equation}
and
\begin{equation}\label{tezpe2}
(\Pi_L-\lambda I)^{*}=\Pi_L^{*}-\overline{\lambda}I,
\end{equation}
where $u\in W^{1,p}$, $v\in W^{1,q}$ and $\frac 1p+\frac 1q=1.$

Similarly, the formally adjoint operators $(\Pi_{L_\pm}-\lambda
I)^{*}$ of $\Pi_{L_\pm}-\lambda I$ are defined by
\begin{equation}\label{tezpe3}
((\Pi_{L_\pm}-\lambda I)^{*}v)(\xi )=dv''(\xi)+cv'(\xi)+L_{\pm}v(\xi
)-\overline{\lambda }v(\xi).
\end{equation}

Throughout the rest of the paper, we let $X:=
L^{p}(\mathbb{R},\mathbb{C}),\,\ 1\leq p\leq \infty$. $\Re X$ is
considered as an ordered Banach space with a positive cone $X_{+}$,
where $\Re X=\{\text{Re}u|u\in X\}$ and $X_{+}=\left\{w\in\Re X
|w\geq 0\right\}$. It is well known that $X_{+}$ is generating,
normal,(see \cite{s1} for more details). For $\varphi\in\Re X$, we
write $\varphi\gneq 0$ if $\varphi\in X_{+}$ and $\varphi\neq 0$,
$\varphi\gg 0$ if $\varphi(\xi)>0$ for each $\xi\in\BR$. An operator
$A:X\rightarrow X$ is called positive if $AX_{+}\subseteq X_{+}$.

\begin{definition}
An operator $A$ is said to be resolvent positive if the resolvent
set $\rho(A)$ of $A$ contains an interval $(\alpha,\infty)$ and
$(\lambda I-A)^{-1}$ is positive for sufficiently large
$\lambda\in\rho(A) \cap \mathbb R$.
\end{definition}

In sequel, we follow \cite{Hen} to define the normal points and the
essential spectrum of an operator $A$ on a Banach space. Namely, a
normal point of $A$ is a complex number in the resolvent set
$\rho(A)$ or an isolated eigenvalue of $A$ with finite multiplicity.
The complement of the set of normal points is called the essential
spectrum of $A$ denoted by $\sigma_{ess}(A)$. We denote the spectral
bound of an operator $A$ by
\[
\frak{s}(A)=\sup\{\text{Re}\lambda: \lambda\in\sigma(A)\}.
\]
We also let $\overline{\iota}=\max\{a^{+}+b^{+},a^{-}+b^{-}\}$ and
$\underline{\iota } =\min \{a^{+}-b^{+},a^{-}-b^{-}\},$ where
$a^{\pm}=f_r(\pm1,\pm1), b^{\pm}=f_s(\pm1,\pm1)$.

\begin{theorem}
Consider the linear operator $\Pi_L$ $:L^p\rightarrow L^p$ defined
by (\ref{tezspa1}), which corresponds to the variational equation of
(1.1) at $U$, that is
\[
(\Pi_Lu)(\xi )=du^{\prime \prime }(\xi )-cu^{\prime
}(\xi)+f_r(U,J*U)u(\xi )+f_s(U,J*U)(J*u)(\xi )
\]
and its formally adjoint operator $\Pi_L^{*}:L^q\rightarrow L^q$
defined by
\[
(\Pi_L^{*}u)(\xi )=du^{\prime \prime }(\xi )+cu^{\prime}(\xi
)+f_r(U,J*U)u(\xi )+(J*f_s(U,J*U)u)(\xi),
\]
where $1\leq p\leq \infty ,\frac 1p+\frac 1q=1.$ $q=\infty$ if
$p=1,$ and $q=1$ if $p=\infty$.  Then

Case $d=0$.

(i) Let $\Omega_{+}=\{\lambda\in\mathbb
C|\text{Re}\lambda>\overline{\iota}\}$ and
$\Omega_{-}=\{\lambda\in\mathbb
C|\text{Re}\lambda<\underline{\iota}\}$. $\lambda $ is a isolated
eigenvalue with finite algebraic multiplicity if $\lambda
\in(\Omega_{+}\cup\Omega_{-})\cap\sigma(\Pi_L)$, Furthermore,
suppose $\psi$ is an eigenfunction corresponding to $\lambda$, then
\begin{equation}\label{stz4}
|\psi(\xi)|\leq C_{\lambda}e^{-\mu|\xi|},\,\,\ \xi\in\mathbb R
\end{equation}
for some positive constants $C_{\lambda}$ and $\mu$.

(ii) $\sigma _{ess}(\Pi_L)\subseteq \{\lambda \in \mathbb
C|\underline{\iota }\leq \text{Re}\lambda \leq \overline{\iota }\}$.

(iii) $\frak{s}(\Pi_L)=\frak{s}(\Pi^{*}_L)=0$ and $0$ is a simple
eigenvalue for both $\Pi_L$ and $\Pi_L^{*}.$

(iv) $\dim \mathcal{N}(\Pi_L)=\dim
\mathcal{N}(\Pi_L^{*})=\text{codim}\mathcal{R}(\Pi_L)=
\text{codim}\mathcal{R}(\Pi_L^{*})=1$. Moreover, $\Pi_L^{*}$ has a
positive eigenfunction $\Psi$ corresponding to the eigenvalue $0$ ,
and
\begin{equation}\label{stz5}
\mathcal{R}(\Pi_L)=\{h\in L^p|\int_{\BR}\Psi(\xi)h(\xi)d\xi =0\}.
\end{equation}

(v) There exist $\varpi >0$ such that the set $\{\lambda \in
\mathbb{C}|\text{Re}\lambda <-\varpi $, or $\text{Re}\lambda
\geq0\,\ \text{and}\,\ \lambda\neq 0\}\subset \rho (\Pi_L)$, where
$\rho (\Pi_L)$ denotes the resolvent set of $\Pi_L.$

Case $d>0.$

(i) Let
$\Xi=\{\lambda\in\mathbb{C}||\text{Im}\lambda|>\sqrt{\overline{\iota}-\text{Re}\lambda}+(b^{-}\wedge
b^{+}) ,\text{Re}\lambda\leq\overline{\iota}\}\cup
\{\text{Re}\lambda>\overline{\iota}\}$. Suppose that
$\lambda\in\Xi\cap\sigma(\Pi_L)$, then $\lambda $ is a isolated
eigenvalue with finite algebraic multiplicity. Furthermore, if
$\psi$ is an eigenfunction corresponding to $\lambda$ then
\[
|\psi(\xi)|\leq C_{\lambda}e^{-\mu|\xi|},\,\,\ \xi\in\BR,
\]
where $C_{\lambda}$ and $\mu$ are positive constants.

(ii) $\sigma _{ess}(\Pi_L)\subseteq\mathbb{C}\setminus\Xi$.

The assertions (iii) and (iv) stated above remain true.

(v)the set $\{\lambda \in \mathbb{C}|\text{Re}\lambda \geq0\,\
\text{and}\,\ \lambda\neq 0\}\subset \rho (\Pi_L)$,
\end{theorem}
\begin{proof}
We shall first prove that $\Pi_L$ is resolvent positive. The proof
for $\Pi_L^{*}$ is same. Let $\widetilde{\lambda}>0$ be sufficiently
large and write $\widetilde{\lambda }=\lambda
^{*}-\underline{\lambda }$ such that $\max_{0\leq U\leq
1}|f_r(U,J*U)|\leq \underline{\lambda }<\infty .$ Then we have
$(\widetilde{\lambda }I-\Pi_L)=(\lambda
^{*}I+c\partial-d\partial^2)-(\underline{\lambda }I+L)$ , where
$(Lv)(\xi):=f_r(U,J*U)v(\xi )+f_s(U,J*U)(J*v)$ and $\partial$
denotes differentiation. According to \cite{Mik}(see section 1.6),
As long as $\lambda ^{*}$ is sufficiently large, $(\lambda
^{*}I+c\partial-d\partial^2)$ is invertible and positive. In
particular, $||(\lambda ^{*}I+c\partial-d\partial^2)^{-1}||\leq
|\frac {k}{\lambda ^{*}}|$ for some $k>0$ provided $d>0$ ,and
$||(\lambda ^{*}I+c\partial)^{-1}||\leq |\frac{c}{\lambda ^{*}}|$.
Since $L$ is a bounded operator, $\widetilde{\lambda }I-\Pi_L$ is
invertible provided $\widetilde{\lambda}$ is sufficiently large.
Therefore, we have
\begin{eqnarray*}
(\widetilde{\lambda }I-\Pi_L)^{-1}
&=&[(\lambda ^{*}I+c\partial-d\partial^2)-(\underline{\lambda }I+L)]^{-1} \\
&=&(\lambda ^{*}I+c\partial-d\partial^2)^{-1}[I-(\underline{\lambda
}I+L)(\lambda
^{*}I+c\partial-d\partial^2)^{-1}] \\
&=&(\lambda ^{*}I+c\partial-d\partial^2)^{-1}\sum_{j=0}^\infty [(\underline{\lambda }%
I+L)(\lambda ^{*}I+c\partial-d\partial^2)^{-1}]^j.
\end{eqnarray*}
Note that $(\underline{\lambda }I+L)$ is positive. So the above
Neumann series is a sum of positive operators and hence is positive.
Clearly, for any $\lambda>\widetilde{\lambda}$, $(\lambda I-\Pi_L)$
is invertible and positive. Hence $\Pi_L$ is resolvent positive.

Next, we prove the statements (i)-(v) for the case that $d=0$. The
proof for case that $d>0$ follows the same lines and shall be
omitted.

By \cite{Mik} again, there exists $\varpi>0$ sufficient large such
that $(\lambda I+c\partial)$ is invertible and $||(\lambda
I+c\partial)^{-1}||\leq |\frac{c}{\lambda}|$ whenever $\lambda\leq
-\varpi$. With the same reasoning used previously, we see that
${\lambda}I-\Pi_L$ is invertible provided $\lambda\leq -\varpi$.
This prove the part of (v).

Notice that both $\Omega_{+}$ and $\Omega_{-}$ are connected open
subsets of $\mathbb C$. Due to Proposition \ref{atep1} and Lemma
\ref{atel1} in Appendix, both $(\Pi_L-\lambda I)$ and
$(\Pi_L^*-\lambda I)$ are semi-Fredholm operators whenever
$\lambda\in\Omega_{-}\cup\Omega_{+}$. Since
$\rho(\Pi_L)\cap\Omega_{+}\neq\emptyset$ and
$\rho(\Pi_L)\cap\Omega_{-}\neq\emptyset$, according the first
paragraph on p243 of \cite{Kato}, the followings hold true:

(a1) $(\Pi_L-\lambda I)$ is Fredholm of index zero if
$\lambda\in\Omega_{-}\cup\Omega_{+}$.

(a2) Suppose that
$\lambda\in\sigma(\Pi_L)\cap(\Omega_{+}\cup\Omega_{-})$, then
$\lambda$ is a isolated eigenvalue with finite algebraic
multiplicity.

Furthermore, Lemma \ref{tezlp1} implies (\ref{stz4}). Therefore (i)
is completed. As a consequence of (i), (ii) is true. Next, we show
(iii) and (iv). Analogously, (a1) and (a2) remain valid for
$\Pi_L^*$. Notice that $0\in\sigma(\Pi_L)$, and so
$\frak{s}(\Pi_L)\geq 0>-\infty$. By \cite{Thi}, the resolvent
positivity yields that $\frak{s}(\Pi_L)\in\sigma(\Pi_L)$. In
particular, $\frak{s}(\Pi_L)\in\sigma(\Pi_L)\cap\Omega_+$ since
$\overline{\iota}<0$. Therefore, (a1) and (a2) imply that
$\text{Ind}(\Pi_L-\frak{s}(\Pi_L)I)=0$ and $\frak{s}(\Pi_L)$ is a
isolated eigenvalue with finite algebraic multiplicity. It follows
from Lemma \ref{atel1} that
\[
\text{codim}\mathcal{R}(\Pi_L^{*}-\frak{s}(\Pi_L)I)=\text{codim}\mathcal{R}
(\Pi_L-\frak{s}(\Pi_L)I)^{*}\geq\dim\mathcal{N}(\Pi_L-\frak{s}(\Pi_L)I)>0.
\]
Consequently, $\frak{s}(\Pi_L)\in\sigma(\Pi_L^*)$. By the resolvent
positivity of $\Pi_L^*$, we infer that
$\frak{s}(\Pi_L^*)\in\sigma(\Pi_L^{*})$ and $\frak{s}(\Pi_L^*)\geq
\frak{s}(\Pi_L)\geq 0$. Moreover, $\frak{s}(\Pi_L^*)$ has a positive
eigenfunction $\Psi$. Suppose that $\frak{s}(\Pi_L^*)>0$. Observe
that $U'(\xi)$ is an eigenvalue of $\Pi_L$ corresponding to
eigenvalue $0$ and hence
$\frak{s}(\Pi_L^*)U'(\xi)\in\mathcal{R}(\Pi_L-\frak{s}(\Pi_L^*)I)$.
Since $(\Pi_L-\frak{s}(\Pi_L^*)I)^{*}=\Pi_L^*-\frak{s}(\Pi_L^*)I$,
it follows from lemma \ref{atel1} in Appendix that
\[
\frak{s}(\Pi_L^*)\int_{\BR}\Psi(\xi)U'(\xi)d\xi=0,
\]
which is impossible since $U'\gg0$. Thus
$\frak{s}(\Pi_L^*)=\frak{s}(\Pi_L)=0$. We now prove the simplicity
of eigenvalue $0$, without loss of generality, we assume that $c>0$.
We first show that $\mathcal{N}(\Pi_L)=\text{span}\{U'\}$. Suppose
this not true, then there is an eigenfunction $\psi$ associated with
eigenvalue $0$ such that $\psi\neq tU'$ for all $t\in\BR$.
Obviously, $\psi\in W^{2,p}(\BR)$. In view of Theorem \ref{ttz3},
$|\psi(\xi)|=O(e^{\lambda^{s}_{+}\xi}),$ as $\xi\rightarrow\infty$,
and $|\psi(\xi)|=O(e^{\lambda^{u}_{-}\xi}),$ as
$\xi\rightarrow-\infty$. Due to the positivity of $U'$, there exist
$t$ such that $tU'+\psi\geq 0$. Let $\overline{t}=\inf\{t\in\BR:
tU'+\psi\geq 0\}$. Obviously, $\overline{t}U'+\psi\neq 0$. Set
$\overline{w}=\overline{t}U'+\psi$ and
$\Sigma=\{\xi\in\BR|\overline{w}(\xi)=0\}$. Note that $\Sigma$ is
not empty by our assumption. Furthermore, $\Sigma$ is a close set
and $\Sigma\backslash\text{inter}\Sigma\neq\emptyset$ . Let
$\xi_0\in\Sigma\backslash\text{inter}\Sigma$. Certainly, for each
$\varepsilon>0$, there is a point $\xi_\varepsilon\in
(\xi_0-\frac{1}{2}\varepsilon,\xi_0+\frac{1}{2}\varepsilon)$ such
that $\overline{w}(\xi_\varepsilon)>0$. Since, for any
$\gamma>\max_{0\leq U\leq 1}|f_r(U,J*U)|$,
$c\overline{w}'+\gamma\overline{w}=(L\overline{w}+\gamma
\overline{w})\gneq0$, simple calculation shows that
\[
\overline{w}(\xi)=\int^{\xi}_{-\infty}e^{-\frac{\gamma}{c}(\xi-\eta)}(L+\gamma
I)\overline{w}(\eta)d\eta,\,\,\ \xi\in\BR.
\]
Clearly, $\overline{w}(\xi)>0$ for any $\xi\geq\xi_0+\varepsilon$.
Thanks to (H1), there exist $a,b$ with $b>a>0$ such that
$(-b,-a)\cup(a,b)\subseteq\text{supp}J$. Since $\varepsilon$ can be
chosen sufficiently small such that $\varepsilon<a$, we find that
$\text{supp}J(\xi_0-\cdot)\cap\text{supp}\overline{w}(\cdot)$
contains a nonempty open interval. Hence $J*\overline{w}(\xi_0)>0$.
On the other hand, $\overline{w}'(\xi_0)=0$ implies that
\[
0\leq
f_s(U,J*U)(J*\overline{w})(\xi_0)\leq(\Pi_{L}\overline{w})(\xi_0)=0.
\]
(H3) forces that $J*\overline{w}(\xi_0)=0$, thus we reach a
contradiction. The contradiction leads to the desired conclusion
that $\mathcal{N}(\Pi_L)=\text{span}\{U'\}$. As mentioned early, we
can similarly show that $\mathcal{N}(\Pi_L)=\text{span}\{U'\}$ for
the case that $d>0$. However, the proof is much simpler. Indeed, we
have
$c\overline{w}'-d\overline{w}''+\gamma\overline{w}=(L\overline{w}+\gamma
\overline{w})\gneq0$, where $\gamma$ is the constant same as one
defined above. Then
\[
\overline{w}(\xi)=\int^{\xi}_{-\infty}e^{\mu_{-}}(L+\gamma
I)\overline{w}(\eta)d\eta+\int^{\infty}_{\xi}e^{\mu_{+}}(L+\gamma
I)\overline{w}(\eta)d\eta,\,\,\ \xi\in\BR,
\]
where $\mu_{\pm}=[c\pm\sqrt{c^2+4d\gamma}](2d)^{-1}$. Thus,
$\overline{w}=\overline{t}U'+\psi\gg 0$, which violates the
definition of $\overline{t}$, and the contradiction yields the
conclusion we need. Next, we show that
$\mathcal{N}(\Pi_{L})^2=\mathcal{N}(\Pi_{L})$, we argue by
contradictions. Let $\Pi_{L}\Phi=t_1U'$ for some $\Phi\in L^{p}$ and
$t_1\in\BR$, that is, $t_1U'\in\mathcal{R}(\Pi_{L})$. Therefore,
$t_1\int_{\BR}\Psi(\eta)U'(\eta)d\eta=0$, which is a contradiction.
with the same reasoning, we can show that
$\mathcal{N}(\Pi^{*}_{L})=\text{span}\{\Psi\}$ and $0$ is also a
simple eigenvalue of $\Pi^{*}_{L}$. Thus, we proved that (iii) and
$\dim \mathcal{N}(\Pi_L)=\dim
\mathcal{N}(\Pi_L^{*})=\text{codim}\mathcal{R}(\Pi_L)=
\text{codim}\mathcal{R}(\Pi_L^{*})=1$. Note that (\ref{stz5}) is
ensured by Lemma \ref{atel1} if $\Pi_L$ is considered in $L^{p}$
with $1\leq p<\infty$. In case $p=\infty$,
$\mathcal{R}(\Pi_L)\subseteq \{h\in L^{\infty}|\int h\Psi=0\}$
implies (\ref{stz5}). Hence (iv) is completed. Certainly,
$\lambda\in\rho(\Pi_L)$ for any $\lambda\in\mathbb{C}$ with
$\text{Re}\lambda>0$. Moreover, by using the arguments similar to
those in \cite{bates2} (see pg 124, also refer to \cite{Volp}), we
can show that $\Pi_L-i\eta$ is injective for any $\eta\in\mathbb R$.
On the other hand, for each $\eta\in\BR$, $\Pi_L-i\eta$ is Fredholm
of index zero. Hence $i\eta\in\rho(\Pi_L)$ for any $\eta\in\BR$ and
(v) is proved.

\end{proof}

\section{Appendix}

Let $\Delta _{L_{\pm}-\lambda I}:$ $\mathbb C\rightarrow\mathbb C$
be the characteristic equations associated with the operators
$\Pi_{L_\pm}-\lambda I$, which is defined by
\[
dz^2-cz-\lambda+a^{\pm }+b^{\pm }\int_{\mathbb R}J(s)e^{-zs}ds.
\]
Let $\Delta^{*} _{L_{\pm}-\overline{\lambda }I }:$ $\mathbb
C\rightarrow\mathbb C$ be the characteristic equations associated
with the adjoint operators $(\Pi_{L_\pm}-\lambda I)^{*}$, which is
defined by
\[
dz^2+cz-\overline{\lambda}+a^{\pm }+b^{\pm }\int_{\mathbb
R}J(s)e^{-zs}ds.
\]

\begin{remark}\label{tezre1}
In light of proposition \ref{tep1}, it is clear that there exist
$\Lambda>0$ such that $\Delta_{L_{+}-\lambda}(z)=0$ ( res
$\Delta_{L_{-}-\lambda}(z)=0$) has no solution in the vertical strip
$\{\lambda\in\mathbb C|-\Lambda\leq Re z\leq\Lambda\}$ provided
$\Delta_{L_{+}-\lambda}(i\eta)\neq 0$ (res
$\Delta_{L_{-}-\lambda}(i\eta)\neq 0$) for any $\eta\in\mathbb R$.
In fact, for any $K>0$, the set $\{-K\leq Re z\leq
K|\Delta_{L_{\pm}-\lambda}(z)=0\}$ is bounded. Since
$\Delta_{L_{\pm}-\lambda}(z)$ is analytic on $\mathbb C$, there are
only finite zeros of $\Delta_{L_{\pm}-\lambda}(z)$ located in the
strip $-K\leq Re z\leq K$, thus, there must exist $\Lambda>0$ such
that $\Delta_{L_{\pm}-\lambda}(z)=0$ has no solution in the vertical
strip $\{\lambda\in\mathbb C|-\Lambda\leq\text{Re}z\leq\Lambda\}$.
In addition, if $\Delta_{L_{+}-\lambda I}(i\eta)\neq0$ (res
$\Delta_{L_{-}-\lambda I}(i\eta)\neq0$) for any $\eta\in\BR$, then
$\Delta_{L_{+}-\overline{\lambda}I}^{*}(i\eta)\neq0$ (res
$\Delta^{*}_{L_{-}-\overline{\lambda}I}(i\eta)\neq0$) for any
$\eta\in\BR$.
\end{remark}

\begin{definition}
The operator $\Pi_{L_+}-\lambda I$ $(\Pi_{L_-}-\lambda I)$ is called
hyperbolic if $\Delta_{L_{+}-\lambda I}(i\eta)\neq0$
$(\Delta_{L_{-}-\lambda I}(i\eta)\neq0)$ for any $\eta\in\BR.$ The
operator $\Pi_L-\lambda I$ is called asymptotic hyperbolic if both
$\Pi_{L_+}-\lambda I$ and $\Pi_{L_-}-\lambda I$ are hyperbolic.
Similarly, the operator $(\Pi_{L_+}-\lambda I)^{*}$
$((\Pi_{L_-}-\lambda I)^{*})$ is called hyperbolic if
$\Delta^{*}_{L_{+}-\lambda I}(i\eta)\neq0$
$(\Delta^{*}_{L_{-}-\lambda I}(i\eta)\neq0)$ for any $\eta\in\BR.$
The operator $(\Pi_L-\lambda I)^{*}$ is called asymptotic hyperbolic
if both $(\Pi_{L_+}-\lambda I)^{*}$ and $(\Pi_{L_-}-\lambda I)^{*}$
are hyperbolic.
\end{definition}

\begin{proposition}\label{tep3}
If $\lambda\in\mathbb C$ such that $\Pi_{L_+}-\lambda I$ is
hyperbolic, then the operator $\Pi_{L_+}-\lambda I$ is an
isomorphism from $W^{1,p}$ onto $L^p$ for $1\leq p\leq\infty$
provided $d=0$. If $d>0$, then $\Pi_{L_+}-\lambda I$ is an
isomorphism from $W^{2,p}$ onto $L^p$ for $1\leq p\leq\infty$. In
each case,the inverse is given by convolution
\begin{equation*}
[(\Pi_{L_+}-\lambda
I)^{-1}h](\xi)=(G^{\lambda}_{L_+}*h)(\xi)=\int_{\mathbb
R}G^{\lambda}_{L_+}(\xi-\eta)H(\eta)d\eta
\end{equation*}
with a function $G^{\lambda}_{L_+}$ which enjoys the estimate
\begin{equation}\label{tezc1}
|G^{\lambda}_{L_+}(\xi)|\leq K'e^{-\alpha|\xi|},\,\,\,\ \xi\in\BR,
\end{equation}
for some constants $k'$ and $\alpha$. Moreover, the same assertion
is valid for $\Pi_{L_-}-\lambda I$.
\end{proposition}
\begin{proof}
Invoking Lemma\ref{tte1}, we only need to show (\ref{tezc1}). By the
remark \ref{tezre1}, there exist $m>0$ such that all zeros of
$\Delta_{L_{+}-\lambda}$ lie outside of the strip
$\{\lambda\in\mathbb C||\text{Re}z|\leq m\}$. We define
\[
k^+=\inf\{\text{Re}z :\Delta_{L_{+}-\lambda}(z)=0,\,\ \text{Re}z>0\}
\]
and
\[
k^-=\sup\{\text{Re}z :\Delta_{L_{+}-\lambda}(z)=0,\,\
\text{Re}z<0\}.
\]
Choose $\varepsilon'>0$ sufficiently small such that
$\Delta_{L_{+}-\lambda}(z)$ only has finite number of zeros in the
strip $k_{-}-\varepsilon'<\text{Re}z\leq\varepsilon'$ and
$\Delta_{L_{+}-\lambda}(z)$ is analytic on
$\text{Re}z=k_{-}-\varepsilon'$. Again, we let
\[
G^{\lambda}_{L_+}(\xi)=\int_{\BR}e^{i\xi\eta}\Delta^{-1}_{L_{+}-\lambda}(i\eta)d\eta.
\]
By the reasoning used in the proof of Lemma \ref{tte1}, we find, for
any $\xi\geq0,$
\begin{eqnarray*}
G^{\lambda}_{L_+}(\xi)
&=&\sum\text{Res}(e^{z\xi}\Delta^{-1}_{L_{+}-\lambda}(z))|_{k_{-}-\varepsilon'<\text{Re}z\leq\varepsilon'}
+e^{(k_{-}-\varepsilon')\xi}\int_{\BR}e^{i\xi\eta}\Delta^{-1}_{L_{+}-\lambda}(k^{-}-\varepsilon'+i\eta)d\eta\\
&=&O(\xi^{i}e^{k_{-}\xi}),\,\,\, \text{as}\,\ \xi\rightarrow\infty.
\end{eqnarray*}
Here we assume that the zero with $\text{Re}z=k_-$ is a $i$th zero
of $\Delta_{L_{+}-\lambda}(z)$ in the strip
$k^{-}-\varepsilon'<\text{Re}z\leq\varepsilon'$. Analogously, we
have
$$G^{\lambda}_{L_+}(\xi)=O(\xi^{j}e^{k_{+}\xi}),\,\,\,
\text{as}\,\ \xi\rightarrow-\infty$$ for some $j\geq 0$. Let
$0<\alpha<\min\{|k_{-}|,|k_{+}|\}$, then (\ref{tezc1}) follows.
\end{proof}

\begin{lemma}\label{tezlp1}
Assume that $\lambda\in\mathbb C$ such that $\Pi_L-\lambda I$ is
asymptotic hyperbolic. Suppose $(\Pi_L-\lambda I)u=h$, where $h\in
L^p$. $u\in W^{1,p}$ when $d=0$, and $u\in W^{2,p}$ when $d\neq0$.
Then
\begin{equation}\label{te23}
|u(\xi)|\leq K_1e^{-\mu |\xi|}||u||_{L^\infty }+K_2\int_{\mathbb
R}e^{-\mu |\xi -\eta |}|h(\eta)|d\eta ,\quad \xi \in \mathbb R.
\end{equation}
\begin{equation}\label{tezap001}
||u||_{W^{1,p}}\leq K_3||u||_{L^p}+K_4||h||_{L^p},\,\,\,\,\ \text{if
}d=0.
\end{equation}
\begin{equation}\label{tezap002}
||u||_{W^{2,p}}\leq K_5||u||_{L^p}+K_6||h||_{L^p},\,\,\,\, \text{if
}d>0.
\end{equation}
Here all the constants $\mu$ and $K_{i} (i=1,\cdot,\cdot,6)$ are
positive and independent of $u$ and $h$.
\end{lemma}
\begin{proof}
Let $G^{\lambda}_{L_+}$ and $G^{\lambda}_{L_-}$ are Green functions
for $\Pi_{L_+}-\lambda I$ and $\Pi_{L_-}-\lambda I$, respectively.
Invoking (\ref{tep3}), we may assume that both $G^{\lambda}_{L_+}$
and $G^{\lambda}_{L_-}$ enjoy the estimate (\ref{tezc1}). Note that
$\Pi_L-\lambda I=\Pi_{L_\pm}-\lambda I+M_{\pm}$, where $M_{\pm}$ are
given by (\ref{tezs1}).

Set
$$(\widetilde{L}^{\lambda}_{\pm}u)(\xi):=(L_{\pm}-\lambda I)u(\xi)+(\vartheta^{\lambda}_{\pm}M_{\pm})(\xi)u(\xi),$$
$$(M^{\lambda}_{\pm}u)(\xi):=((1-\vartheta^{\lambda}_{\pm})M_{\pm
})(\xi)u(\xi),$$  where $\vartheta^{\lambda}_{\pm}(\xi)$ are the
unit step functions similar to (\ref{te2.2}). Then $(\Pi_L-\lambda
I)u=h$ is equivalent to
\begin{equation}\label{te19}
\widetilde{\Pi}^{\lambda}_{L_{\pm}}u=h-M^{\lambda}_{\pm}u,
\end{equation}
where
$\widetilde{\Pi}^{\lambda}_{L_{\pm}}=du''-cu'+(\widetilde{L}^{\lambda}_{\pm}u)$.
Choose $\vartheta^{\lambda}_{\pm}$ such that
$\vartheta^{\lambda}_{\pm}M_{\pm}(\xi)$ satisfy the conditions of
Proposition \ref{pte2}, here we may assume that
$\vartheta^{\lambda}_{\pm}$ have their jump points at $\pm\sigma$
respectively. Then $\widetilde{\Pi}^{\lambda}_{L_{\pm}}$ are
isomorphisms from $W^{1,p}(W^{2,p})$ onto $L^{p}$ for $1\leq p\leq
\infty$. Let $\widetilde{G}^{\lambda}_{L_{\pm}}$ be the Green
functions for $\widetilde{\Pi}^{\lambda}_{L_{\pm}}$, then
Proposition \ref{pte2} yields
\[
|\widetilde{G}^{\lambda}_{L_{\pm}}(\xi)|\leq
C_{\lambda}e^{-\mu|\xi|}
\]
for some positive constants $C_\lambda,\mu$. Consequently, we have
either
\begin{eqnarray*}
u(\xi) &=&\int_{\mathbb R}\widetilde{G}^{\lambda}_{L_{+}}(\xi,\eta
)[-M_{+}^\lambda u(\eta )]d\eta +\int_{\mathbb
R}\widetilde{G}^{\lambda}_{L_{+}}(\xi ,\eta
)h(\eta )d\eta  \\
&=&\int_{-\infty }^\sigma \widetilde{G}^{\lambda}_{L_{+}}(\xi,\eta
)[-M_{+}^\lambda u(\eta )]d\eta +\int_{\mathbb
R}\widetilde{G}^{\lambda}_{L_{+}}(\xi ,\eta
)h(\eta )d\eta  \\
&\leq &\int_{-\infty }^\sigma e^{-\mu|\xi -\eta |}||M_{+}^\lambda
||||u||_{L^\infty }d\eta +\int_{-\infty }^\infty e^{-\mu|\xi -\eta
|}h(\eta )d\eta
\end{eqnarray*}
or
\begin{eqnarray*}
u(\xi) &=&\int_{\mathbb R}\widetilde{G}^{\lambda}_{L_{-}}(\xi,\eta
)[-M_{-}^\lambda u(\eta )]d\eta +\int_{\mathbb
R}\widetilde{G}^{\lambda}_{L_{-}}(\xi
,\eta )h(\eta )d\eta  \\
&\leq &\int_{-\sigma }^\infty e^{-\mu|\xi -\eta |}||M_{-}^\lambda
||||u||_{L^\infty }d\eta +\int_{-\infty }^\infty e^{-\mu |\xi -\eta
|}h(\eta )d\eta.
\end{eqnarray*}
Thus, (\ref{te23}) follows.  Now we define
\begin{equation}\label{te22}
(\Lambda _\omega v)(\xi)=dv''(\xi)-cv'(\xi)-\omega v(\xi),
\end{equation}
then we have
\begin{equation}\label{te24}
\Lambda _\omega u=(\lambda-\omega)u-Lu+h,
\end{equation}
where $(Lv)(\xi):=L(\xi)v(\xi)$. As long as $\omega
>0$ is sufficiently large $\Lambda _\omega ^{-1}:$ $L^p\rightarrow W^{1,p}$ $(W^{2,p})$
exists and satisfies $||\Lambda _\omega
^{-1}u||_{W^{1,p}}\leq\omega_p||u||_{L^{p}}$ provided $d=0$.
($||\Lambda _\omega ^{-1}u||_{W^{1,p}}\leq\omega_p||u||_{L^{p}}$ if
$d>0$), where $\omega_p$ only depends on $d,c,\omega$ and $p$. Since
$L$ is a bounded operator in $L^{p}$, the desired conclusions
(\ref{tezap001}) and (\ref{tezap002}) follow.
\end{proof}

\begin{remark}\label{tezr2}
In virtue of remark \ref{tezre1} and Lemma \ref{tezlp1},
$(\Pi_{L}-{\lambda}I)^{*}$ is asymptotically hyperbolic if and only
if $\Pi_{L}-\lambda I$ is asymptotically hyperbolic, where
$\lambda\in\mathbb C$. Suppose that $\Pi_{L}-\lambda I$ is
asymptotically hyperbolic and $\mathcal{N}(\Pi_{L}-\lambda I)$ is
nonempty, where $\mathcal{N}(\Pi_{L}-\lambda I)$ is the kernel of
$\Pi_{L}-\lambda I$. Let $\phi\in \mathcal{N}(\Pi_{L}-\lambda I)$,
then Lemma \ref{tezlp1} implies that $\phi$ decays exponentially at
infinity. Clearly, the same conclusion holds true for
$(\Pi_{L}-\lambda I)^{*}$ provided it has nonempty kernel.
\end{remark}

\begin{proposition}\label{tezpp2}
Assume that $\lambda \in \mathbb{C}$ such that $(\Pi_{L}-\lambda
I)^{*}$ is asymptotically hyperbolic. Suppose for some $p$ that
there are bounded sequences $u_n\in W^{1,p}$ ( $W^{2,p}$ when $d>0$)
and $h_n\in L^p$ such that $(\Pi_{L}-\lambda I)^{*}u_n=h_n$ and
$h_n\rightarrow \overline{h}$ in $L^{p}$. Then there exists a
subsequence $u_{n^{\prime }}$ and some $\overline{u}\in W^{1,p}$ (
$W^{2,p}$) such that $u_{n^{\prime }}\rightarrow \overline{u}$ in
$W^{1,p}$ ( $W^{2,p}$) and $(\Pi_{L}-\lambda
I)\overline{u}=\overline{h}.$
\end{proposition}
\begin{proof}
Let the operator $N_{\pm}(\xi):L^{p}\rightarrow L^{p}$ defined by
\[
N_{\pm}(\xi)u(\xi)=[f_{s}(U(\xi),J*U(\xi))-f_s(\pm1,\pm1)]u(\xi)
+\int_{\BR}J(\xi-\eta)[f_{s}(U(\eta),J*U(\eta))-f_s(\pm1,\pm1)]u(\eta)d\eta.
\]
Also set
$$(\widehat{L}^{\lambda}_{\pm}u)(\xi):=(L_{\pm}-\lambda I)u(\xi)+N_{\pm}(\xi)(\vartheta^{\lambda}_{\pm}u
)(\xi),$$
$$(N^{\lambda}_{\pm}u)(\xi):=N_{\pm
}(\xi)((1-\vartheta^{\lambda}_{\pm})u)(\xi),$$  where
$\vartheta^{\lambda}_{\pm}(\xi)$ are same as these defined in Lemma
\ref{tezlp1}.  Let
$\widehat{\Pi}^{\lambda}_{L_{\pm}}=du''-cu'+(\widehat{L}^{\lambda}_{\pm}u)$.Then,
with the same reasoning, we draw the desired conclusion.

\end{proof}

\begin{lemma}\label{atel1}
Suppose that $\Pi_{L}-\lambda I$ is asymptotically hyperbolic. Then
the operator $\Pi_L-\lambda I:L^p\rightarrow L^p$ is Fredholm
operator for each $p$ $(1\leq p<\infty )$. Furthermore, the range
$\mathcal{R}(\Pi_L-\lambda I)$ is given by
\[
\mathcal{R}(\Pi_{L}-\lambda I)=\{h\in L^p|\int_\BR\overline{w}(\xi
)h(\xi )d\xi =0,\quad w(\xi )\in \mathcal{N}((\Pi_{L}-\lambda
I)^{*}))\}.
\]
In particular,
\[
\dim\mathcal{N}(\Pi_{L}-\lambda
I)^{*}=\text{codim}\mathcal{R}(\Pi_{L}-\lambda I),\,\,\
\dim\mathcal{N}(\Pi_{L}-\lambda
I)=\text{codim}\mathcal{R}(\Pi_{L}-\lambda I)^{*},
\]
\[
\text{Ind}(\Pi_{L}-\lambda I)=-\text{Ind}(\Pi_{L}-\lambda I)^{*}.
\]
In case $p=\infty$, $\Pi_{L}-\lambda I$ is semi-Fredholm operator.
Additionally, the operator $\Pi_{L}-\lambda I$ is Fredholm if $J$
has compact support.
\end{lemma}

\begin{proof}
The proof of this lemma is very similar to that of Theorem A in
\cite{Mall}, we shall therefore only sketch the proof. As usual, we
shall only give the proof for the case that $d=0$ since the proof
for the case that $d>0$ can be completed analogously. We start to
show that the unit ball
\[
\mathcal{B}=\{u\in W^{1,p}|u\in \mathcal{N}(\Pi_{L}-\lambda
I),||u||_{W^{1,p}}\leq 1\}
\]
in $\mathcal{N}(\Pi_{L}-\lambda I)$ $\subset W^{1,p}$ is compact,
and hence we can conclude $\dim \mathcal{N}(\Pi_{L}-\lambda
I)<\infty $. It is worth pointing out that
$\mathcal{N}(\Pi_{L}-\lambda I)$ is independent of $p.$ Indeed, this
can be inferred from the remark \ref{tezr2}. Now, we choose any
sequence $u_n\in \mathcal{B},$ then by Proposition \ref{tezpp1} with
$h_n=0$, there exists a subsequence $u_{n^{^{\prime }}}\rightarrow
u^{*}$ in $W^{1,p}$ for some $u^{*}$ with $(\Pi_{L}-\lambda
I)u^{*}=0$. Therefore, $u^{*}\in \mathcal{B}$ and $\mathcal{B}$ is
compact.

Next we let $p$ be fixed and we show that
$\mathcal{R}(\Pi_{L}-\lambda I)$ is closed. Let $h_n\in
\mathcal{R}(\Pi_{L}-\lambda I)$ $\subseteq L^p$ such that
$h_n\rightarrow h^{*}$ in $L^p$, then we need to show that $h^{*}\in
\mathcal{R}(\Pi_{L}-\lambda I).$ Let $\mathcal{C}$ $\subseteq
W^{1,p}$ be a closed subspace complement of
$\mathcal{N}(\Pi_{L}-\lambda I)$, that is,
$W^{1,p}=\mathcal{N}(\Pi_{L}-\lambda I)\oplus \mathcal{C}.$ Clearly,
there exists a sequence $u_n\in \mathcal{C}$ such that
$(\Pi_{L}-\lambda I)u_n=h_n.$ As shown in \cite{Mall},
$||u_n||_{W^{1,p}}$ must be bounded, hence Proposition \ref{tezpp1}
implies that there exists $u^{*}\in \mathcal{C}$ such that
$(\Pi_{L}-\lambda I)u^{*}=h^{*}.$ This prove the closeness of
$\mathcal{R}(\Pi_L-\lambda I).$ Therefore, $\Pi_{L}-\lambda I$ is
semi-Fredholm.

Now, we assume that $1\leq p<\infty$, in order to prove
$(\Pi_{L}-\lambda I)$ is Fredholm, it suffices to show that
$\mathcal{R}(\Pi_{L}-\lambda I)$ has finite codimensions in $L^p.$
To this end, we let $\mathcal{N(}(\Pi_{L}-\lambda
I)^{*})^0_{p}\subseteq L^p$ denote
\[
\mathcal{N(}(\Pi_{L}-\lambda I)^{*})^0_{p}=\{h\in
L^p|\int_\BR\overline{w}(\xi)h(\xi )d\xi =0,\quad w(\xi)\in
\mathcal{N}((\Pi_{L}-\lambda I)^{*})\}.
\]
From remark \ref{tezr2} , we see that $(\Pi_{L}-\lambda I)^{*}$ is
also asymptotically hyperbolic, and hence Proposition \ref{tezpp1}
together above arguments imply that
$\dim\mathcal{N}((\Pi_{L}-\lambda I)^{*})<\infty $, where
$\mathcal{N}((\Pi_{L}-\lambda I)^{*})\subseteq W^{1,q}.$ Certainly,
codim$\mathcal{N(}(\Pi_{L}-\lambda I)^{*})^0_{p}=\dim
\mathcal{N}((\Pi_{L}-\lambda I)^{*})<\infty $. To complete the
proof, we show that $\mathcal{N(}(\Pi_{L}-\lambda
I)^{*})^0_{p}=\mathcal{R}(\Pi_{L}-\lambda I)$. In view of
(\ref{tezpe}), we have $\mathcal{R}(\Pi_{L}-\lambda
I)\subseteq\mathcal{N}((\Pi_{L}-\lambda I)^{*})^0_{p}$. Assume for
some $p$ $(1\leq p<\infty ) $ that $\mathcal{R}(\Pi_{L}-\lambda
I)\neq \mathcal{N(}(\Pi_{L}-\lambda I)^{*})^0_{p}$, then there
exists $v^{*}\in \mathcal{R}(\Pi_{L}-\lambda I)^{\bot}$ and
$\int_\BR\overline{v^{*}(\xi )}h(\xi )d\xi \neq 0$ for some $h\in
\mathcal{N(}(\Pi_{L}-\lambda I)^{*})^0_{p}$, where
$\mathcal{R}(\Pi_{L}-\lambda I)^{\bot}=\{v\in
L^q|\int_\BR\overline{v(\xi)}g(\xi )=0,\quad g\in
\mathcal{R}(\Pi_{L}-\lambda I)\}$. Clearly, $v^{*}\notin
\mathcal{N}((\Pi_{L}-\lambda I)^{*}).$ On the other hand, we have
\[
\int_\BR\overline{v(\xi)}(\Pi_{L}-\lambda I)u(\xi)d\xi=0,u\in
W^{1,p}.
\]
Choose any $\chi \in C^\infty (\BR,\mathbb C)$ with compact support
and set $u=\overline{\chi}$ . By taking the complex conjugates, we
find
\begin{eqnarray*}
0 &=&\int_{\BR}v(\xi )\overline{(\Pi_{L}-\lambda I)\overline{\chi (\xi)}}d\xi  \\
&=&d\int_{\BR}\chi(\xi)v''(\xi )d\xi+c\int_{\BR}\chi (\xi)v'(\xi
)d\xi+\int_{\BR}\chi (\xi)(L^{*}(\xi)-\overline{\lambda })v(\xi)d\xi\\
&=&\int_{\BR}\chi(\xi)((\Pi_{L}-\lambda I)^{*}v)(\xi)d\xi.
\end{eqnarray*}
This indicates that $v$ solves the adjoint equation in the sense of
distributions and $v\in W^{1,q}$. Thus
$v\in\mathcal{N}((\Pi_{L}-\lambda I)^{*}).$ This contradiction
establishes that $\mathcal{R}(\Pi_{L}-\lambda
I)=\mathcal{N(}(\Pi_{L}-\lambda I)^{*})^0_{p}.$ for $1\leq
p<\infty$.

For the case of $p=\infty $, once again, (\ref{tezpe}) implies that
$\mathcal{R}(\Pi_{L}-\lambda I)\subseteq\mathcal{N(}(\Pi_{L}-\lambda
I)^{*})^0_{\infty}.$ Suppose that $J$ has compact support. We need
to show that $\mathcal{R}(\Pi_{L}-\lambda
I)\supseteq\mathcal{N(}(\Pi_{L}-\lambda I)^{*})^0_{\infty}.$ We
start to show that every $h\in\mathcal{N(}(\Pi_{L}-\lambda
I)^{*})^0_{\infty}$ can be written as $h=h_1+h_2,$ where
$h_1\in\mathcal{R}(\Pi_{L}-\lambda I),$ and $h_1=h$ whenever
$|\xi|\geq \tau$ for some positive constant $\tau$.
 Certainly, $h_2\in\mathcal{N(}(\Pi_{L}-\lambda I)^{*})^0_{\infty}$
and $h_2$ has compact support. Therefore,
$h_2\in\mathcal{N(}(\Pi_{L}-\lambda I)^{*})^0_{p}$. As shown before,
$h_2\in\mathcal{R}(\Pi_{L}-\lambda I)$. Hence we have
$h\in\mathcal{R}(\Pi_{L}-\lambda I)$ as desired. To construct $h_1$,
we let $w_{\pm}=(\widetilde{\Pi}^{\lambda}_{L_{\pm}})^{-1}h,$ Where
$\widetilde{\Pi}^{\lambda}_{L_{\pm}}$ are isomorphisms given in
Lemma \ref{tezlp1}. Due to Lemma \ref{tezlp1}, there exists
$\sigma>0$ such that $(\Pi_{L}-\lambda I)w_{+}=h$ for any
$\xi\geq\sigma$, while $(\Pi_{L}-\lambda I)w_{-}=h$ for any
$\xi\leq-\sigma$. Now, we let
$w^{*}(\xi)=m(\xi)w_{+}(\xi)+(1-m(\xi))w_{-}(\xi)$, where
$m:\BR\rightarrow\BR^{+}$ is a $C^2$ function such that $m(\xi)=0$
for $\xi\leq0$, and $m(\xi)=1$ for $\xi\geq1$. Since $J$ has compact
support , the direct computation shows that $(\Pi_{L}-\lambda
I)w^{*}=h$ provided $|\xi|$ is sufficiently large. Choose
$h_1=(\Pi_{L}-\lambda I)w^{*}$, as required. Hence the proof is
completed.
\end{proof}

Now we set
$\overline{\iota}=\max\{a^{+}+b^{+},a^{-}+b^{-}\},\underline{\iota }
=\min \{a^{+}-b^{+},a^{-}-b^{-}\}$.

$\Omega_{+}=\{\lambda\in\mathbb{C}|\text{Re}\lambda>\overline{\iota}\},$
$\Omega_{-}=\{\lambda\in\mathbb{C}|\text{Re}\lambda<\underline{\iota}\}$,
$\Xi=\{\lambda\in\mathbb{C}|\text{Re}\lambda>\overline{\iota}\}
\cup\{\lambda\in\mathbb{C}|\text{Im}z|>c^2\sqrt{\overline{\iota}-\text{Re}z}+(b^{+}\wedge
b^{-})\}$.

\begin{proposition}\label{atep1}
If $\lambda\in\Xi$ then $\lambda I-\Pi_L$ is asymptotic hyperbolic
when $d>0$.  In case $d=0$, the same conclusion also holds true if
$\lambda\in\Omega_{+}\cup\Omega_{-}.$
\end{proposition}
\begin{proof}
It is sufficient to prove that $\Delta _{L_{\pm}-\lambda }(i\eta
)\neq 0$ for any $\eta \in\BR$ provided $\text{Re}\lambda
>\overline{\iota}.$ Note that $\Delta _{L_{\pm }-\lambda }(i\eta
)=0$ if and only if
\[
ic\eta +d\eta ^2+(\lambda -a^{\pm})=b^{\pm }\int_{\BR}J(s)e^{-i\eta
s}ds.
\]

First we note that
\[
d\eta ^2+\text{Re}\lambda -a^{\pm }>b^{\pm }\geq |b^{\pm
}\int_{\BR}J(s)e^{-i\eta s}ds|,\quad \eta \in\BR
\]
when $\text{Re}\lambda >\overline{\iota},$ hence $\Delta
_{L_{\pm}-\lambda}(i\eta)\neq 0$ for all $\eta\in\BR,$ provided
$\text{Re}\lambda >\overline{\iota}.$ Now, suppose $\text{Re}\lambda
\leq\overline{\iota},$ it is easy to see that
\[
c^{-2}(\text{Im}z-b^{\pm}\int_{\BR}J(s)\sin(-\eta
s)ds)^2+\text{Re}z-a^{\pm}>b^{\pm}\int_{\BR}J(s)\cos(\eta s)ds,\,\,\
\eta\in\BR,
\]
whenever
$|\text{Im}z|>c^2\sqrt{\overline{\iota}-\text{Re}z}+(b^{+}\wedge
b^{-})$. In case $d=0,$ we still have $\Delta
_{L_{\pm}-\lambda}(i\eta)\neq 0$ if
$\text{Re}\lambda>\overline{\iota}$. Moreover, $\text{Re}\lambda
<\underline{\iota }$ implies
\[
\text{Re}\lambda -a^{\pm }<-b^{\pm }\leq \text{Re}(b^{\pm
}\int_{\BR}J(s)e^{-i\eta s}ds),\quad \eta \in\BR.
\]
Therefore, the desired conclusion follows.
\end{proof}

\end{document}